\pdfoutput=1                      

\documentclass[leqno,10pt]{amsart}

\usepackage[accsupp]{axessibility}    

\usepackage{amsmath,color}
\usepackage[all,cmtip]{xy}
\usepackage{mathrsfs}
\usepackage{geometry}
\newtheorem{introtheorem}{Theorem}

\newtheorem{introcorollary}[introtheorem]{Corollary}

\newtheorem{theorem}{Theorem}[section]
\newtheorem{corollary}{Corollary}

\newtheorem{lemma}[theorem]{Lemma}
\newtheorem{proposition}{Proposition}

\theoremstyle{definition}
\newtheorem{definition}[theorem]{Definition}
\newtheorem{remark}{Remark}
\newtheorem{example}{Example}

\newcommand{\ep}{\varepsilon}

\newcommand{\ind}{\mathrm{ind}}
\newcommand{\Mbar}{\overline{M}}

\newcommand{\RR}{\mathbb R}
\newcommand{\PP}{\mathbb P}

\newcommand{\sphere}{\mathbb S}
\newcommand{\HH}{\mathbb H}
\newcommand{\gbar}{\bar g}
\newcommand{\metric}{\langle\;,\;\rangle}

\newcommand{\etabar}{\bar\eta}

\newcommand{\nablabar}{\overline{\nabla}}

\newcommand{\II}{\mathrm{I\hspace{-1pt}I}}
\newcommand{\Hvec}{\mathbf{H}}
\newcommand{\Hessbar}{\overline{\Hess}}
\newcommand{\Ricbar}{\overline{\Ric}}
\newcommand{\disp}{\displaystyle}
\newcommand{\loc}{\mathrm{loc}}
\newcommand{\vol}{\mathrm{vol}}

\renewcommand{\div}{\diver}

\newcommand{\di}{\mathrm{d}}

\newcommand{\ra}{\rightarrow}

\DeclareMathOperator{\trace}{tr}
\DeclareMathOperator{\Hess}{Hess}
\DeclareMathOperator{\Lip}{Lip}
\DeclareMathOperator{\Ric}{Ric}
\DeclareMathOperator{\diver}{div}
\DeclareMathOperator{\sgn}{sgn}

\DeclareMathOperator{\dist}{dist}

\title{Remarks on mean curvature flow solitons in warped products}

\author{Giulio Colombo}
\address{Dipartimento di Matematica, Universit\`a degli Studi di Milano, 20133 Milano, Italy}
\curraddr{}
\email{giulio.colombo@unimi.it}
\thanks{}

\author{Luciano Mari}
\address{Scuola Normale Superiore, Piazza dei Cavalieri, 7, 56124 Pisa, Italy}
\curraddr{}
\email{luciano.mari@sns.it}
\thanks{}

\author{Marco Rigoli}
\address{Dipartimento di Matematica, Universit\`a degli Studi di Milano, 20133 Milano, Italy}
\curraddr{}
\email{marco.rigoli@unimi.it}
\thanks{}

\begin{document}

\maketitle

\centerline{\emph{Dedicated to Patrizia Pucci on her $65^{\mathrm{th}}$ birthday}}

\begin{abstract}
	We study some properties of mean curvature flow solitons in general Riemannian manifolds and in warped products, with emphasis on constant curvature and Schwarzschild type spaces. We focus on splitting and rigidity results under various geometric conditions, ranging from the stability of the soliton to the fact that the image of its Gauss map be contained in suitable regions of the sphere. We also investigate the case of entire  graphs.
\end{abstract}

\bigskip

	

\bigskip

\tableofcontents

\section{Introduction}

The mean curvature flow (MCF) is a smooth map $\Psi : [0,T] \times M^m \ra \overline{M}^n$ between Riemannian manifolds $M$, $\overline{M}$ such that $\Psi_t = \Psi(t, \cdot)$ is an immersion for every $t$, and satisfying 
$$
	\frac{\partial \Psi}{\partial t} = m\mathbf{H}_{t},
$$
with $\mathbf{H}_{t}$ the normalized mean curvature of $M_t = \Psi_t(M)$ with the induced metric. We are interested in a MCF that moves along the flow $\Phi$ of a smooth vector field $X \in \mathfrak{X}(\overline{M})$, namely, for which there exists a reparametrization $t \mapsto s(t)$ and a flow $\eta : [0,T] \times M \ra M$ of some tangential vector field on $M$ satisfying
\begin{equation}\label{self_similar}
	\Psi(t,x) = \Phi\big( s(t), \Psi_0(\eta(t,x)) \big).
\end{equation}
Typically, interesting $X$ are conformal or Killing fields, so the induced MCF is self-similar at every time. As expected, this study is motivated by the importance that self-shrinkers, self-translators and self-expanders have in the classical MCF in Euclidean space, and although self-similar solutions of the MCF on general manifolds do not arise from blow-up procedures, nevertheless they provide suitable barriers and thus their study is important to grasp the behaviour of the MCF in an ambient space that possesses some symmetries. This is the original motivation for the present work. \par
Differentiating \eqref{self_similar} in $t$ one deduces that $m\mathbf{H}_t = s'(t) X^\perp$, with $X^\perp$ the component of $X$ orthogonal to $M_t$. Restricting to the time slice $t=0$, this motivates the following definition recently proposed in \cite{AdeLR}.

\begin{definition} \label{intro:def:soliton}
	An isometric immersion $\psi:M^m\to\Mbar^{n+1}$ is a mean curvature flow soliton with respect to $X\in\mathfrak X(\Mbar)$ if there exists $c \in \RR$ (hereafter called the soliton constant) such that
	\begin{equation} \label{intro:soliton}
	cX^{\bot} = m\Hvec
	\end{equation}
	along $\psi$.
\end{definition}

\noindent (Observe that in \cite{AdeLR} $\Hvec$ is not normalized, that is, $\Hvec=\trace\II$). In particular, for $c=0$ the map $\psi$ is a minimal submanifold. In the statements of the results below we shall pay attention to specify the sign of the soliton constant: we feel convenient not to include $c$ in the field $X$ because, later on, we will focus on some ``canonical" fields in $\overline{M}$, and inserting the scale and time-orientation in a separate constant $c$ makes things easier. 

In what follows we typically, but not exclusively, consider codimension $1$ solitons in a warped product 
\begin{equation} \label{intro:wprod}
\Mbar = I\times_h\PP,
\end{equation}
where $I\subseteq\RR$ is an interval, $h:I\to\RR^+=(0,+\infty)$ is a smooth function and $\overline{M}$ is endowed with the warped metric
\begin{equation} \label{intro:wmetric}
\gbar = \di t^2 + h(t)^2\metric_{\PP}
\end{equation}
where $\metric_{\PP}$ is the metric on $\PP$ and $t$ is the standard coordinate on $I$. In this setting there exists a natural closed conformal vector field $X$ on $\Mbar$ coming from a potential and given by
\begin{equation} \label{intro:defX}
X = h(t)\partial_t = \overline{\nabla} \bar \eta,
\end{equation}
where
\begin{equation} \label{B:defetabar}
\etabar(x)=\int_{t_0}^{\pi_I(x)} h(s) \di s,
\end{equation}
with $\pi_I : I \times_h \PP \ra I$ the projection onto the first factor and $t_0 \in I$ a fixed value. If $X = \overline{\nabla} \bar \eta$, for future use we set
\begin{equation}\label{def_eta_intro}
\eta = \bar \eta \circ \psi \ \ : \ \ M \ra \RR.
\end{equation}	
Recall that a vector field $X$ is said to be closed when the dual $1$-form is closed. One essential feature of this setting is that it enables us to derive, for a mean curvature flow soliton with respect to $X$ as in (\ref{intro:defX}), a manageable form of some differential equations (see for instance (\ref{B:etaequation}) below) that will be used in our analytical investigation. The second order operator naturally appearing in those formulas is the drifted Laplacian
\begin{equation}\label{def_driftedLapla}
\Delta_{-cX^{\top}}=\Delta+c\langle\nabla\;\cdot\;,X^{\top}\rangle,
\end{equation}
that for gradient fields $X = \nablabar \etabar$, as the one in \eqref{intro:defX}, becomes the weighted Laplacian 
\begin{equation}\label{def_weightedLapla}
\Delta_{-c\eta}= \Delta+c \langle\nabla\;\cdot\;, \nabla \eta\rangle.
\end{equation}
For more details, and the very general form of the equations mentioned above we refer again to \cite{AdeLR}. Let us also recall, for terminology, that when $\Mbar=I\times\PP$ is simply a product ($h\equiv1$) and $X=\partial_t$ the soliton is called translational. 

\bigskip

\noindent \textbf{Notation and agreements.} Throughout the paper, all manifolds $M$,  $\overline{M}$ will be assumed to be connected; $B_R(o) \subset M$ will denote the geodesic ball of radius $R>0$ centered at a fixed origin $o \in M$ and $\partial B_R(o)$ its boundary. If no confusion may arise we omit to write the origin $o$.

\bigskip

\begin{example}\label{ex_hyper}
	The hyperbolic space $\HH^{m+1}$ of curvature $-1$ admits representations as the warped products $\RR \times_{e^t} \RR^m$ and $\RR \times_{\cosh t} \HH^ m$. In the first case, the slices of the foliation for constant $t$ are horospheres, whereas in the second case they are hyperspheres.
\end{example}

\begin{example}\label{ex_schwarz}
	Given a mass parameter $\mathfrak{m}>0$, the Schwarzschild space is defined to be the product
	\begin{equation}\label{schwarz}
		\overline{M}^{m+1} = (r_0(\mathfrak{m}), \infty)\times \PP^m \qquad \text{with metric} \qquad \metric_{\overline{M}} = \frac{\di r^2}{V(r)} + r^2 \metric_\PP, 
	\end{equation}
	where $\PP$ is a (compact) Einstein manifold with $\Ric_\PP = (m-1)\metric_\PP$, 
	\begin{equation}\label{def_Vr_schwarz}
		V(r) = 1-2 \mathfrak{m} r^{1-m}
	\end{equation}
	and $r_0(\mathfrak{m})$ is the unique positive root of $V(r)=0$. Its importance lies in the fact that the manifold $\RR \times \overline{M}$ with the Lorentzian static metric $-V(r) \di t^2 + \metric_{\overline{M}}$ is a solution of the Einstein field equation in vacuum with zero cosmological constant. In a similar way, given $\bar \kappa \in \{1,0,-1\}$ the ADS-Schwarzschild space (with, respectively, spherical, flat or hyperbolic topology according to whether $\bar \kappa = 1,0,-1$) is the manifold \eqref{schwarz} with respectively
	\begin{equation}\label{def_Vr_ADS}
		V(r) = \bar \kappa + r^2 - 2 \mathfrak{m} r^{1-m}, \qquad \Ric_\PP = (m-1)\bar \kappa \metric_{\PP}, 
	\end{equation}
	and $r_0(\mathfrak{m})$, as before, the unique positive root of $V(r)=0$. They generate static solutions of the vacuum Einstein field equation with negative cosmological constant, normalized to be $-m(m+1)/2$. The products \eqref{schwarz} can be reduced in the form $I \times_h \PP$ with metric \eqref{intro:wmetric} via the change of variables
	\begin{equation}\label{def_tr_sch}
		t = \int_{r_0(\mathfrak{m})}^r \frac{\di \sigma}{\sqrt{V(\sigma)}}, \qquad h(t) = r(t), \qquad I = \RR^+.
	\end{equation}
	Note that $r_0(\mathfrak{m})$ is a simple root of $V(r)=0$, so the integral is well defined, and that $1/\sqrt{V(r)} \not \in L^1(\infty)$ so $I$ is an unbounded interval.
	A similar example for charged black-hole solutions of Einstein equation is the Reissner-Nordstr\"om-Tangherlini space, where \eqref{schwarz} holds with
	$$
		V(r) = 1-2 \mathfrak{m} r^{1-m} + \mathfrak{q}^2 r^{2-2m}, \qquad \PP = \sphere^{m},
	$$
	$\mathfrak{q} \in \RR$ being the charge and $|\mathfrak{q}| \le \mathfrak{m}$.
\end{example}

The paper is naturally divided into three parts. First we study codimension $1$ solitons in the presence of a Killing vector field $B$ on the target space $\Mbar$. The relevant differential equation we analyse is satisfied by $u=\langle B,\nu\rangle$, $\nu$ a unit normal to $\psi:M\to\Mbar$. We will often assume that $\psi$ is \emph{two-sided}, that is, that the normal bundle is trivial, so there exists \emph{globally defined} unit normal $\nu$. If $\overline{M}$ is orientable, this is equivalent to require that $M$ is oriented.\par
By way of example, we apply our results to the case of a self-shrinker of $\RR^{m+1}$, that is, a soliton with respect to the position vector field of $\RR^{m+1}$ with soliton constant $c<0$: the next result was recently shown by Ding, Xin, Yang \cite{DXY}, however our proof in Theorem \ref{B:thm:euclid} below is different, and especially it follows from a strategy that works in greater generality. It serves the purpose to illustrate our technique. For the sake of clarity, we remark that a continuous map $\psi : M \to N$ between differentiable manifolds is said to be proper if $\psi^{-1}(K)$ is a compact subset of $M$ for every compact subset $K$ of $N$. When $M$ and $N$ are complete Riemannian manifolds, this is equivalent to saying that $\psi^{-1}(B)$ is bounded in $M$ for every bounded subset $B$ of $N$.

\begin{introtheorem} \label{intro:thm:euclid}
	Let $\psi:M^m\to\RR^{m+1}$ be an oriented, connected, complete proper self-shrinker whose spherical Gauss map $\nu$ has image contained in a closed hemisphere of $\sphere^m$. Then either $\psi$ is totally geodesic (that is, $\psi(M)$ is an affine hyperplane) or $\psi$ is a cylinder over some connected, complete proper self-shrinker of $\RR^m$.
\end{introtheorem}

As we shall see, the result heavily depends on the parabolicity of the weighted Laplacian $\Delta_{-c\eta}$ on a self-shrinker. However, if $\Delta_{-c\eta}$ is possibly non-parabolic we can still conclude a similar  rigidity for $M$ when the Gauss map is contained in a hemisphere, by using some results in oscillation theory (cf. \cite{BMR}). Suppose that $\psi : M \to I \times_h \PP$ is a complete soliton with respect to $h(t) \partial_t$ and with constant $c$, and define $\bar \eta, \eta$ as in \eqref{B:defetabar}, \eqref{def_eta_intro}. Recall that an end of $M$ (with respect to a compact set $K$) is a connected component of $M \backslash K$ with noncompact closure. Fix an origin $o \in M$ and an end $E$, and define
\begin{equation}\label{def_VE}
	v_E(r) = \int_{\partial B_r \cap E}e^{c\eta},
\end{equation}
where $B_r$ is the geodesic ball of $M$ centered at $o$. Given $a(x)\in C^0(M)$ we introduce the weighted spherical mean
\begin{equation} \label{B:ameansdef}
	A_E(r) = \frac{1}{v_E(r)}\int_{\partial B_r \cap E}a(x)e^{c\eta}
\end{equation}
If $\frac{1}{v_E(r)}\in L^1(\infty)$, the rigidity of $M$ will depend on how much $A_E(r)$ exceeds the critical curve
\begin{equation}\label{def_critical}
	\chi_E(r) = \left\{2v_E(r)\int_r^{+\infty}\frac{\di s}{v_E(s)}\right\}^{-2} \quad \text{for all } r>R.
\end{equation}
For explanations on the origin of $\chi_E$ and its behaviour, as well as for other geometric applications, we refer to Chapters 4, 5, 6 in \cite{BMR}. Here we limit ourselves to observe that, in the minimal case $c=0$ and for the Euclidean space $\RR^m$ (for which $v_E(r) = \omega_{m-1}r^{m-1}$), 
$$
	\chi_E(r) = \frac{(m-2)^2}{4r^2}
$$
is the potential in the famous Hardy inequality on $\RR^m$. In fact, there is a tight relation between $\chi_E(r)$ and Hardy weights, explored in \cite{BMR}.

\par
We are ready to state our theorem. Results in this spirit can be found in \cite{HOS} (for $m=2$), \cite{FR} (for $m \ge 3$ and a compact setting) and \cite[Thm. 5.36]{BMR}.

\begin{introtheorem} \label{B:thm:inv}
	Let $\PP^m$ be a Riemannian manifold of dimension $m\geq 3$, $I \subseteq \RR$ an open interval and let $\psi:M^m\to\Mbar=I\times\PP^m$ be an oriented, connected, complete translating mean curvature flow soliton with respect to the vector field $X=\partial_t$, with soliton constant $c$ and global unit normal $\nu$. Suppose that there exists an end $E$ of $M$ such that $\langle\nu,\partial_t \rangle$ does not change sign on $E$, and that, having fixed an origin $o\in M$ and set $v_E, A_E$ as in \eqref{def_VE}, \eqref{B:ameansdef} with the choice
	\begin{equation}\label{baretaax}
		\bar \eta = \pi_I, \qquad a(x) = \left( \Ricbar(\nu,\nu)+ |\II|^2 \right)(x),
	\end{equation}
	either
	\begin{equation} \label{B:ass1}
		\begin{array}{lll}
			(i) & v_E^{-1} \not \in L^1(\infty), & \; \disp \lim_{r\to+\infty}\int_{E \cap B_r}e^{c(\pi_I \circ \psi)}\left( \Ricbar(\nu,\nu)+ |\II|^2 \right) =+\infty, \quad \text{or} \\[0.5cm]
			(ii) & v_E^{-1} \in L^1(\infty), & \; A_E \ge 0 \ \ \text{on } [2R, \infty) \quad \text{and} \\[0.2cm]
			& & \; \disp \limsup_{r \ra \infty} \int_{2R}^r \Big(\sqrt{A_E(s)}- \sqrt{\chi_E(s)}\Big) \di s = +\infty,
		\end{array}
	\end{equation}
	with $\chi_E$ as in \eqref{def_critical} if $v_E^{-1} \in L^1(\infty)$. Then, $\psi$ is a minimal hypersurface that splits as a Riemannian product $\RR \times \Sigma$ with $\psi_{\Sigma} : \Sigma^{m-1} \to \PP^m$ a minimal hypersurface itself, and the soliton vector $X$ is tangent to $\psi$ along the straight lines $\psi(\RR \times \{y\})$, $y \in \Sigma$. Furthermore, under this isometry $\psi: \RR \times \Sigma \ra \Mbar$ can be written as the cylinder
	\begin{equation}\label{splitting}
		\psi(t,y) = \psi(0,y) + tX.
	\end{equation}
\end{introtheorem}

In the second part we look at solitons as extrema of an appropriate weighted volume functional and we consider the stable case according to Section 3. The stability condition often implies a rigidity for $M$ provided that suitable geometric quantities like $\mathbf{H}$, the second fundamental form $\II$ or the umbilicity tensor 
$$
\Phi = \II - \metric_M \otimes \mathbf{H} 
$$
have controlled $L^p$ norms. This is the case of the next result, Theorem \ref{stab:thm:umbilicity} below, one of the main of the present paper. Its proof depends on Lemma \ref{lem_kato}, that classifies Codazzi tensors whose traceless part satisfies the equality in Kato inequality: this might be of independent interest. We here write the result in a slightly simplified form, and we refer the reader to Theorem \ref{stab:thm:umbilicity} below for a more detailed statement of item 2.

\begin{introtheorem} \label{teo_umbilico_intro}
	Let $\psi:M^m\to\Mbar^{m+1}=I\times_h\PP$ be a connected, complete, stable mean curvature flow soliton with respect to $X=h(t)\partial_t$ with soliton constant $c$. Assume that $\Mbar$ is complete and has constant sectional curvature $\bar\kappa$, with
	$$
		ch'(\pi_I\circ\psi)\leq m\bar\kappa \quad \text{on } \, M.
	$$
	Let $\Phi=\II-\metric_M\otimes\Hvec$ be the umbilicity tensor of $\psi$ and suppose that
	$$
		|\Phi|\in L^2(M,e^{c\eta})
	$$
	with $\eta$ defined as above. Then one of the following cases occurs:
	\begin{enumerate}
		\item $\psi$ is totally geodesic (and if $c \neq 0$ then $\psi(M)$ is invariant by the flow of $X$), or
		\item $M$, $\Mbar$ are flat manifolds, $M$ has linear volume growth and $\psi$ lifts to a grim reaper immersion $\hat \psi : \RR^m \ra \RR^{m+1}$.
	\end{enumerate}
\end{introtheorem}

\begin{remark}
	Particular cases of Theorem \ref{teo_umbilico_intro}, for a translating soliton in Euclidean space $\psi : M^m \ra \RR^{m+1}$, were proved in \cite{IR} under stronger conditions: namely, the authors required either $|\II| \in L^2(M, e^{c\eta})$, or $|\Phi| \in L^2(M, e^{c\eta})$ and the hypothesis that $H$ does not change sign (Thms. A and B therein, respectively). Note that the case of the grim reaper soliton is excluded in their setting since $|\Phi| \not \in L^2(M, e^{c\eta})$. We emphasize that, in case 2, necessarily $\overline{M}$ shall be a nontrivial quotient of $\RR^{m+1}$.
\end{remark}

Applying the above theorem to solitons in the sphere, we deduce the next

\begin{introcorollary} \label{intro:crl:sphere}
	Let $\psi:M^m \ra \sphere^{m+1}$ be a complete, connected mean curvature flow soliton with respect to the conformal field $X=\sin r\,\partial_r$ centered at some point $p\in\sphere^{m+1}$, where $r$ is the distance function from $p$ in $\sphere^{m+1}$. Assume that the soliton constant $c$ satisfies
	$$
		c \leq m
	$$
	and that $|\Phi|\in L^2(M)$. Then $\psi$ is unstable. 
\end{introcorollary}

Suppose that $\psi:M^m\to\Mbar^{m+1}=I\times_h\PP$ is a soliton in a warped product as above, with soliton constant $c$. In \cite{AdeLR} the authors introduced the soliton function 
$$
	\zeta_c(t) = mh'(t) +ch(t)^2,
$$
whose geometric meaning is explained in detail in Proposition 7.1 of \cite{AdeLR}. In what follows it suffices to observe that 
\begin{equation}\label{rem_slicesoliton}
	\begin{array}{c}
		\text{the leaf $\{\bar t\}\times\PP^m$ is a soliton} \\[0.2cm]
		\text{with respect to $h(t)\partial_t$, with constant $c$}
	\end{array} \qquad \Longleftrightarrow \qquad \zeta_c(\bar t)=0.
\end{equation}

For instance realizing $\HH^{m+1}$ as the product $\RR\times_{e^t}\RR^m$ via horospheres, given a soliton $\psi:M^m\to\HH^{m+1}$ with respect to $e^t\partial_t$ with constant $c$, the corresponding soliton function reads
$$
	\zeta_c(t) = me^t+ce^{2t}.
$$
In the last part of the paper, we investigate the case where the soliton $\psi = \psi_u$ is the graph of a function $u : \PP \to \RR$ and we complement recent results in \cite{BMPR}. In particular, we prove the following rigidity theorem for Schwarzschild type spaces in Example \ref{ex_schwarz}.

\begin{introtheorem} \label{intro:sch}
	There exists no entire graph in the Schwarzschild and ADS-Schwarz\-schild space (with spherical, flat or hyperbolic topology) over a complete $\PP$ that is a soliton with respect to the conformal field $r \sqrt{V(r)} \partial_r$ with constant $c \ge 0$, where $V(r)$ is the corresponding potential function defined in either \eqref{def_Vr_schwarz} or \eqref{def_Vr_ADS}.
\end{introtheorem}

We remark that solitons with $c < 0$ in such spaces, in general, do exist, see Example \ref{ex_schwarz_constant} below. Regarding the hyperbolic space, we prove

\begin{introcorollary} \label{intro:crl:hspaceslice}
	Let $\HH^{m+1} = \RR \times_{e^t} \RR^m$ be the hyperbolic space of curvature $-1$ foliated by horospheres and let $X = e^t\partial_t \in \mathfrak X(\HH^{m+1})$. Then
	\begin{enumerate}
		\item there exist no entire graphs $\psi_u : \RR^m \to \HH^{m+1}$ over the horosphere $\RR^m$ which are mean curvature flow solitons with respect to $X$ with $c \ge 0$;
		\item if $\psi_u : \RR^m \to \HH^{m+1}$ is an entire graph over the horosphere $\RR^m$ which is a mean curvature flow soliton with respect to $X$ with soliton constant $c<0$, then 
		\[
		\mathrm{dist}\left( \psi_u(\RR^m),  \left\{t = \log  \left(- \frac m c \right) \right\} \right) = 0.
		\]
	\end{enumerate}
\end{introcorollary}

\begin{remark}
	For $c=0$, that is, for minimal graphs, item 1.~has been proved with an entirely different technique (the moving plane method) by M.P. Do Carmo and H.B. Lawson \cite{docarmolawson}. In their work, they also consider graphs over hyperspheres $\psi_u : M \ra \RR \times_{\cosh t} \HH^m$: this case will be treated in Proposition \ref{prp_hyperspheres} below.
\end{remark}

Part of the techniques used in the proof of the results have been introduced in some of our previous papers. Thus, in order to help the reader and for the sake of clarity, we recall, when needed, the original or appropriately modified statements of the theorems we use in the arguments below.

\section{Invariance by a $1$-parameter group of isometries}

The aim of this section is to study the geometry of codimension $1$ mean curvature flow solitons in the presence of a Killing vector field $B$ on $\Mbar$. We begin with a computational result.

\begin{proposition} \label{B:prp:uequation}
	Let $\psi:M^m\to\Mbar^{m+1}$ be a two-sided, mean curvature flow soliton with respect to $X\in\mathfrak X(\Mbar)$ with soliton constant $c$. Let $\nu$ be a chosen unit normal vector field on $M$ and let $B\in\mathfrak X(\Mbar)$ be a Killing field. Set $u=\langle\nu,B\rangle$; then
	\begin{equation} \label{B:uequation}
		\Delta_{-cX^{\top}}u + |\II|^2u + \Ricbar_{-cX}(\nu,\nu)u = \langle[cX,B],\nu\rangle
	\end{equation}
	where $\Delta_{-cX^{\top}}$ is as in \eqref{def_driftedLapla} and $\Ricbar_{-cX}$ is the modified Bakry-Emery Ricci tensor field on $\Mbar$ given by
	\begin{equation} \label{B:RiccX}
		\Ricbar_{-cX} = \Ricbar + \frac{1}{2}\mathcal L_{-cX}\metric_{\Mbar}
	\end{equation}
	with $\mathcal L_{-cX}$ the Lie derivative along $-cX$.
\end{proposition}

\begin{remark}
	The immersion $\psi : M \to \Mbar$ is said to be two-sided if the induced normal bundle on $M$ is trivial, so that there exists a global unit normal vector field $\nu$ as in the statement of the Proposition. 
\end{remark}

\begin{proof}
	We fix the index ranges $1\leq i,j,\dots\leq m$, $1\leq a,b,\dots\leq m+1$ and we let $\{\theta^a\}$ be an oriented Darboux coframe along $\psi$ with corresponding Levi-Civita connection forms $\{\theta^a_b\}$. If $\{e_a\}$ is the dual frame then
	$$
		B = B^ae_a
	$$
	and the fact that $B$ is Killing is equivalent to the skew symmetry
	\begin{equation} \label{B:killing}
		B^a_b + B^b_a = 0.
	\end{equation}
	Note that, since we are dealing with an oriented Darboux coframe, along $\psi$, $\nu=e_{m+1}$ so that
	$$
		u = B^{m+1}.
	$$
	Differentiating $u$ we obtain
	\begin{equation}
		\di u = \di B^{m+1} = B^{m+1}_j\theta^j - B^t\theta^{m+1}_t = (B^{m+1}_i - B^th_{ti})\theta^i,
	\end{equation}
	that is,
	\begin{equation}
		u_i = B^{m+1}_i - B^th_{ti}.
	\end{equation}
	In particular
	\begin{equation} \label{B:gradu}
		\nabla u = B^{m+1}_ie_i - B^sh_{si}e_i.
	\end{equation}
	Thus
	\begin{align*}
		u_{ik}\theta^k & = \di u_i - u_t\theta^t_i = \di B^{m+1}_i - \di (B^th_{ti}) - (B^{m+1}_t - B^sh_{st})\theta^t_i \\
		& = (B^{m+1}_{ik} - B^s_ih_{sk} - B^s_kh_{si} - B^sh_{sik} - uh_{ks}h_{si})\theta^k.
	\end{align*}
	Using (\ref{B:killing}) it follows that
	\begin{equation} \label{B:laplacianu1}
		\Delta u = -B^th_{tkk} - |\II|^2u + B^{m+1}_{kk}.
	\end{equation}
	We recall the commutation relations
	\begin{equation} \label{B:commrel}
		B^{m+1}_{ki}=-B^k_{m+1,i}=-B^k_{i,m+1}-B^a\bar R_{ak,m+1,i}
	\end{equation}
	and Codazzi equations
	\begin{equation} \label{B:codazzi}
		h_{tik}=h_{tki}-\bar R_{m+1,tik}.
	\end{equation}
	Tracing (\ref{B:commrel}) and (\ref{B:codazzi}) with respect to $k$ and $i$ and using (\ref{B:killing}) we have
	\begin{align*}
		B^{m+1}_{kk} & = -{\Ricbar}(B,\nu) \\
		-B^th_{tkk} & = -B^th_{ktk} = -B^th_{kkt} + B^t\bar R_{m+1,ktk}.
	\end{align*}
	Thus, substituting in (\ref{B:laplacianu1}) we finally obtain
	\begin{equation} \label{B:laplacianu2}
		\Delta u + |\II|^2u + \Ricbar(\nu,\nu)u = -B^sh_{kks}.
	\end{equation}
	Next, from the soliton equation (\ref{intro:soliton}),
	\begin{equation} \label{B:soliton}
		h_{kk} = cX^{m+1}
	\end{equation}
	where $X=X^ie_i+X^{m+1}\nu$ along $\psi$. On the other hand on $M$ we have
	$$
		X^{m+1}_i\theta^i = \di X^{m+1} + X^k\theta^{m+1}_k = \di X^{m+1} + X^kh_{ki}\theta^i.
	$$
	Differentiating equation (\ref{B:soliton}) we thus obtain
	$$
		h_{kki}=c(X^{m+1}_i-X^th_{ti})
	$$
	and using (\ref{B:killing}) we have
	$$
		X^iB^{m+1}_i-B^aX^{m+1}_a=X^aB^{m+1}_a-B^aX^{m+1}_a=\langle\nablabar_X B-\nablabar_B X,\nu\rangle=\langle[X,B],\nu\rangle.
	$$
	Inserting into (\ref{B:laplacianu2}) and using (\ref{B:killing}), (\ref{B:gradu}) we infer
	\begin{align*}
		\Delta u + |\II|^2u + \Ricbar(\nu,\nu)u & = -cB^iX^{m+1}_i+cX^th_{ti}B^i \\
		& = -c(X^iB^{m+1}_i-X^th_{ti}B^i)+c(X^iB^{m+1}_i-B^iX^{m+1}_i) \\
		& = -cX^i(B^{m+1}_i-h_{ti}B^t) \\
		& \phantom{=\;} +c(X^aB^{m+1}_a-B^aX^{m+1}_a+B^{m+1}X^{m+1}_{m+1}) \\
		& = -c\langle\nabla u,X^{\top}\rangle + \langle[cX,B],\nu\rangle+cX^{m+1}_{m+1}u
	\end{align*}
	and observing that
	$$
		cX^{m+1}_{m+1}=\langle\nablabar_{\nu}cX,\nu\rangle=\frac{1}{2}(\mathcal L_{cX}\metric_{\Mbar})(\nu,\nu)=-\frac{1}{2}(\mathcal L_{-cX}\metric_{\Mbar})(\nu,\nu)
	$$
	we obtain (\ref{B:uequation}).
\end{proof}

\begin{remark} \label{B:rmk:uequationgrad}
	In the particular case where $X=\nablabar\etabar$ for some $\etabar\in C^{\infty}(\Mbar)$, the tensor field $\Ricbar_{-cX}$ is the Bakry-Emery Ricci tensor \begin{equation} \label{B:Ricceta}
		\Ricbar_{-c\etabar} = \Ricbar - c\Hessbar(\etabar)
	\end{equation}
	and the differential operator $\Delta_{-cX^{\top}}$ is the weighted Laplacian $\Delta_{-c\eta}$, where $\eta=\etabar\circ\psi$. Hence, equation (\ref{B:uequation}) becomes
	\begin{equation} \label{B:uequationgrad}
		\Delta_{-c\eta}u+|\II|^2u+\Ricbar_{-c\etabar}(\nu,\nu)u=\langle[cX,B],\nu\rangle.
	\end{equation}
\end{remark}

The idea is now to couple equation (\ref{B:uequationgrad}) with a second equation of the same type and holding independently of the first. Towards this aim we consider the case of an oriented mean curvature flow soliton $\psi:M^m\to\Mbar^{m+1}$ with respect to $X=\nablabar\etabar$ for some $\etabar\in C^{\infty}(\Mbar)$ with constant $c$ and we compute $\Delta_{-c\eta}\eta$, where $\eta=\etabar\circ\psi$. We begin by computing $\Delta\eta$. The chain rule formula gives
\begin{equation} \label{B:laplacianeta}
	\Delta\eta = \sum_{i=1}^m\Hessbar(\etabar)(d\psi(e_i),d\psi(e_i))+\langle\nablabar\etabar, m \mathbf{H}\rangle_{\Mbar}
\end{equation}
where $\{e_i\}_{i=1}^m$ is a local orthonormal frame on $M$. On the other hand $\psi$ is a mean curvature flow soliton and thus
$$
	\langle m\Hvec,\nablabar\etabar\rangle_{\Mbar} = c\langle X^{\bot},\nablabar\etabar\rangle_{\Mbar} = c\langle\nablabar\etabar^{\bot},\nablabar\etabar^{\bot}\rangle = c(|\nablabar\etabar|^2-|\nabla\eta|^2).
$$
Using (\ref{B:laplacianeta}) we can therefore write
$$
	\Delta_{-c\eta}\eta = \Delta\eta+c|\nabla\eta|^2=\overline{\Delta}\etabar-\Hessbar(\etabar)(\nu,\nu)+c|\nablabar\etabar|^2,
$$
In the particular case $\Mbar=I\times_h\PP$ and $X=h(t)\partial_t$, we have $X=\nablabar\etabar$ with $\etabar$ defined by \eqref{B:defetabar} for some $t_0\in I$. Thus $X$ is conformal and
$$
	\Hessbar(\etabar) = h'(\pi_I) \metric_{\Mbar}.
$$
From $|\nablabar\etabar|^2=h(\pi_I)^2$ we therefore obtain
\begin{equation} \label{B:etaequation}
\Delta_{-c\eta}\eta = m h'(t)+ch(t)^2
\end{equation}
along $\psi$.\par

In the next results, under suitable geometric assumptions we will infer a Liouville theorem for the function $\langle \nu, B \rangle$, with $B$ Killing. The geometric rigidity that follows from it depends on the next version of a splitting theorem investigated, in the intrinsic setting, by Tashiro
\cite{Ta} and Obata \cite{Ob}.

\begin{lemma}\label{lem_splitting}
	Let $\psi : M^m \ra \Mbar^{m+1}$ be an orientable, connected, complete immersed hypersurface, and let $f \in C^\infty(\Mbar)$ satisfy
	$$
		\overline{\Hess} f = \mu \metric_{\Mbar},
	$$
	for some $\mu \in C^\infty(\Mbar)$. If $\nablabar f$ is tangent to $M$ and does not vanish identically, then the set $\mathscr{C}_f \subset M$ of critical points of the restriction of $f$ to $M$ consists of at most two points,  $M \backslash \mathscr{C}_f$ is isometric to the warped product
	$$
		I \times_h \Sigma, \qquad \text{with metric} \quad \metric = \di t^2 + h(t)^2 g_\Sigma,
	$$
	$I \subset \RR$ is an open interval, $\Sigma$ is a (regular) level set of $f$ with induced metric $g_\Sigma$, $\partial_t = \nabla f/|\nabla f|$, $f$ only depends on $t$ and $h(t) = f'(t)$. In particular, in these coordinates $\mu(t,y) = f''(t)$. Moreover, 
	\begin{itemize}
		\item[-] if $|\mathscr{C}_f|=0$, then $I = \RR$,
		\item[-] if $|\mathscr{C}_f|=1$, then $I = (t_1,\infty)$ or $(-\infty, t_2)$, for some $t_1 < 0$, $t_2>0$, and $M$ is globally conformal to either $\RR^m$ or the hyperbolic space $\mathbb{H}^m$,
		\item[-] if $|\mathscr{C}_f|=2$, then $I = (t_1,t_2)$ for some $t_1 < 0 < t_2$, and $M$ is globally conformal to the sphere $\sphere^m$.
	\end{itemize}
	Furthermore, in coordinates $(t,y) \in I \times \Sigma$ the immersion $\psi$ writes as
	\begin{equation}\label{splitting_tashi}
		\psi(t,y) = \Phi\big( t, \psi(0,y) \big) 
	\end{equation}
	where $\Phi$ is the flow of $\nablabar f/|\nablabar f|$ on $\Mbar \backslash \mathscr{C}_{\bar f}$.
\end{lemma}

\begin{proof}
	First observe that $X= \nabla f$ is a gradient, conformal field also on $M$: 
	$$
		\Hess f(V,W) = \langle \nabla_V \nabla f, W \rangle = \langle \bar \nabla_V \nablabar f, W \rangle = \overline{\Hess} f(V,W).
	$$
	The completeness of $M$ together with Lemma 1.2 and Section 2 in \cite{Ta} guarantee that $\mathscr{C}_f$ has at most $2$ points, and that $M \backslash \mathscr{C}_f$ splits as indicated. The identity \eqref{splitting_tashi} is a straightforward consequence.
\end{proof}

Recalling that a self-shrinker in $\RR^{m+1}$ is a mean curvature flow soliton with respect to the position vector field of $\RR^{m+1}$ with soliton constant $c<0$, we have the following result reported in the Introduction.

\begin{theorem}[\cite{DXY}, Thm. 3.1]\label{B:thm:euclid}
	Let $\psi:M^m\to\RR^{m+1}$ be an oriented, connected, complete proper self-shrinker whose spherical Gauss map $\nu$ has image contained in a closed hemisphere of $\sphere^m$. Then, either $\psi$ is totally geodesic or $\psi$ is a cylinder over a proper, connected, complete self-shrinker in $\RR^m$.
\end{theorem}

\begin{remark}\label{rem_chengzhou}
	We remark that, by \cite{CZ}, the properness of a complete immersed self-shrinker $\psi : M^m \ra \RR^{m+1}$ is equivalent to any of the following properties:
	\begin{itemize}
		\item[(i)] $M$ has Euclidean volume growth of geodesic balls,
		\item[(ii)] $M$ has polynomial volume growth of geodesic balls,
		\item[(iii)] $M$ has finite weighted volume:
		$$
		\int_M e^{ c \frac{|\psi|^2}{2}} < \infty.
		$$
	\end{itemize}
	Note that we used a different normalization with respect to \cite{CZ}, where the soliton field is half of the position vector field and thus $|\psi|^2/4$ appears in $(iii)$.
\end{remark}

\begin{proof}
	Since $\psi$ is proper, the pre-image $K = \psi^{-1}(0)$ is compact. We then consider $\psi:M \backslash K \to\Mbar$ with $\Mbar=\RR^{m+1}\setminus\{0\}=\RR^+\times_h\sphere^m$ where $h(t)=t$. The restriction $\psi$ is a mean curvature flow soliton with respect to $X=t\partial_t$, with $c<0$. Now, up to an inessential additive constant, for $y\in\Mbar$, $\etabar(y)=\frac{1}{2}t(\pi_{\RR^+}(y))^2$ and since $\psi$ is proper
	$$
		\eta(x)=\etabar(\psi(x))\to+\infty \quad \text{as } x \to\infty \text{ in } M.
	$$
	From (\ref{B:etaequation}) and $c<0$ we then deduce
	$$
		\Delta_{-c\eta}\eta(x)=m+2c\eta(x)\to-\infty \quad \text{as } x\to\infty \text{ in } M.
	$$
	The above conditions imply, by Theorem 4.12 of \cite{AMR}, that the operator $\Delta_{-c\eta}$ is parabolic on $M$. Since by assumption $\nu(M)$ is contained in a closed hemisphere of $\sphere^m$, there exists a unit vector $b\in\sphere^m$ such that $\langle b,\nu\rangle \ge 0$. We extend $b$ to a parallel field $B$ on $\RR^{m+1}$. Let $\{x^a\}_{a=1}^{m+1}$ be a Cartesian chart on $\RR^{m+1}$ centered at $o$ such that $B=\partial_{x^1}$. Then $X=x^a\partial_{x^a}$ and we have
	$$
		[X,B]=[x^a\partial_{x^a},\partial_{x^1}]=x^a[\partial_{x^a},\partial_{x^1}]-(\partial_{x^1}x^a)\partial_{x^a}=-\partial_{x^1}=-B.
	$$
	Since $\Hessbar(\etabar)=\metric_{\Mbar}$, using (\ref{B:uequationgrad}) and (\ref{B:Ricceta}) we deduce that $u=\langle B,\nu\rangle$ is a non-negative solution of
	\begin{equation}\label{neg_firsteigen}
		\Delta_{-c\eta}u+|\II|^2u-cu=\langle-cB,\nu\rangle=-cu
	\end{equation}
	that is,
	$$
		\Delta_{-c\eta}u = - |\II|^2u \le 0.
	$$
	Thus, the $\Delta_{-c\eta}$-parabolicity implies that $u$ is a non-negative constant. If $u>0$, then $|\II|^2\equiv 0$. Otherwise, from $0 \equiv u = \langle \nu, B \rangle$ we deduce that $B$ is tangent to $M$, and by Lemma \ref{lem_splitting} it follows that $M$ is a cylinder $\RR \times \Sigma$ and that the immersion factorizes. The fact that $\Sigma$ is a self-shrinker in $\RR^m$ is a straightforward consequence.
\end{proof}

Note that we have used equation (\ref{B:etaequation}) to show that on the complete manifold $M$ the operator $\Delta_{-c\eta}$ is parabolic. However, from Theorem 4.1 of \cite{AMR}, we know that a second sufficient condition is given by
$$
	\frac{1}{\int_{\partial B_r}e^{c\eta}} \notin L^1(+\infty),
$$
where $\partial B_r$ is the boundary of the geodesic ball $B_r$ centered at some fixed origin $o \in M$.
That, in turn, is implied, see for instance Proposition 1.3 of \cite{RS}, by
\begin{equation} \label{parab_peso}
	\int_{B_r} e^{c\eta} = \mathcal{O}(r^2) \qquad \text{as } \, r \ra \infty,
\end{equation}
in particular, by
\begin{equation} \label{B:L1parab}
	e^{c\eta}\in L^1(M).
\end{equation}
In the case of self-shrinkers for the position vector field in $\RR^{m+1}$ we readily recover condition $(iii)$ in Remark \ref{rem_chengzhou}. However, the equivalence in Remark \ref{rem_chengzhou} is a specific feature of the self-shrinkers case.

The situation for \emph{translational} solitons with Gauss map contained in a hemisphere is less rigid, since there are examples of translational solitons in $\RR^{m+1}$ that are convex, entire graphs over a totally geodesic hyperplane (cf. \cite{AWu, CSS, W}). However, if the image of the Gauss map is compactly contained into an open hemisphere, the immersion is still totally geodesic,  see \cite{BS}. When the Gauss map is just contained in a  closed hemisphere, then we can still obtain a rigidity theorem provided that the second fundamental form grows sufficiently fast. This will be investigated in Theorem \ref{B:thm:euclidtrans} below.

Nevertheless, for more general classes of warped products with nonnegative curvature one can still obtain interesting results. By way of example, we consider the following

\begin{proposition} \label{B:prp:prod}
	Let $\Mbar= \RR\times\PP$ be a product manifold with
	\begin{equation}\label{riccP}
		\Ric_\PP\geq 0
	\end{equation}
	and let $\psi:M^m\to\Mbar^{m+1}$ be a two-sided, connected, complete  translational soliton with respect to $\partial_t$ with soliton constant $c$. Suppose that the angle function $\langle\partial_t,\nu\rangle$ does not change sign on $M$ and that $M$ satisfies
	\begin{equation}\label{vol_subq}
		\int_{B_r} e^{c (\pi_I \circ \psi)} = o(r^2) \qquad \text{as } r \ra \infty
	\end{equation}
	for some (therefore, any) choice of an origin $o \in M$. Then, either
	\begin{itemize} 	
		\item[$(i)$] $M$ is a slice, or
		\item[$(ii)$] $\partial_t$ is tangent to $M$ and $M$ splits as a cylinder in the $t$-direction, or
		\item[$(iii)$] $\psi$ is totally geodesic, there exists a Riemannian covering map $\pi_0 : \PP_0 \ra \PP$ with $\PP_0 = \RR \times \Sigma$ for some compact $\Sigma$, and there exists a totally geodesic immersion $\psi_0 : M \ra \RR \times \PP_0$ such that $\psi = (\mathrm{id} \times\pi_0) \circ \psi_0$. Moreover, in coordinates $(s,y) \in \RR \times \Sigma$, $\psi_0(M)$ writes as the graph of a function $v : \PP \ra \RR$ of the type $v(s,y) = as+b$, for some $a \in \RR\backslash\{0\}$, $b \in \RR$.
	\end{itemize}
\end{proposition}

\begin{remark}
	As a particular case, \eqref{vol_subq} is satisfied under the assumption that $M$ is contained into a half-space $[a, \infty) \times \PP$, $c \le 0$ and $M$ has less than quadratic \emph{Riemannian} volume growth: $\vol(B_r) = o(r^2)$ as $r \ra \infty$. 
\end{remark}

\begin{proof}
	The non-negativity of $\Ric_\PP$ is equivalent to that of the ambient Ricci curvature $\overline{\Ric}$. Set $u = \langle \nu, \partial_t \rangle$. Since $h\equiv 1$,
	$$
		\Ricbar_{-c\etabar}(\nu,\nu) = \Ricbar(\nu,\nu) \geq 0,
	$$
	thus (\ref{B:uequationgrad}) becomes
	\begin{equation} \label{B:uequationprod}
		\Delta_{-c\eta}u + (|\II|^2+\Ricbar(\nu,\nu))u = 0.
	\end{equation}
	Up to changing $\nu$ we can suppose $0 \le u \le 1$, thus either $u \equiv 0$ and $M$ is invariant by $\partial_t$ (so $(ii)$ holds by Lemma \ref{lem_splitting}), or $0 < u \le 1$ on $M$. From \eqref{vol_subq} and the discussion in the previous pages we deduce that $\Delta_{-c\eta}$ is parabolic, thus $u$ is a positive constant and $\II\equiv0$ follows from \eqref{B:uequationprod}. The soliton equation then implies $c=0$. If $u \equiv 1$ then $M$ is a slice, and $(i)$ occurs. Otherwise, $u \equiv k \in (0,1)$, so the projection $\pi_\PP : \RR \times \PP \ra \PP$ restricted to $M$ is a local diffeomorphism. However, $M$ might not be a global graph over $\PP$. To overcome the problem, we let $\PP_0$ be $M$ endowed with the metric $\pi^* \metric_\PP$, and note that $\pi : \PP_0 \ra \PP$ is a covering because of the completeness of $\PP$. Introducing the covering $\RR \times \PP_0$ of $\RR \times \PP$, we can lift $\psi$ to a (totally geodesic) soliton $\psi_0 : M \ra \RR \times \PP_0$, and in this way $M$ becomes \emph{globally} the graph of a function $v : \PP_0 \ra \RR$. The second fundamental form of the graph is $\II = \pm (1+ |Dv|^2)^{-1/2}\Hess_\PP v$, where $D$ is the connection on $\PP$, thus $\Hess_\PP v \equiv 0$. Furthermore, $Dv \neq 0$ at every point of $\PP_0$ (because $u < 1$), hence $Dv$ realizes a Riemannian splitting $\PP_0 = \RR \times \Sigma$ for some totally geodesic $\Sigma$. Moreover, in realizing the splitting one observes that $v$ is linear in the $\RR$ coordinate and independent of $\Sigma$ (cf. Step 3 in the proof of Thm. 9.3 in \cite{PRS}, or \cite[Thm. 5]{FMV}). We claim that $\Sigma$ is compact. Otherwise, having observed that \eqref{riccP} implies $\Ric_\Sigma \ge 0$, if $\Sigma$ were not compact the Calabi-Yau volume estimate would imply that the volume of geodesic balls in $\Sigma$ grows at least linearly in the radius, hence 
	$$
		\vol_{\PP_0}(B_r) \ge B r^2 \qquad \text{for } \, r \ge 1,
	$$
	for some constant $B >0$. However, from \eqref{vol_subq} and $c=0$, we infer
	$$
		o(r^ 2) = \int_{B_r} e^{c(\pi_I \circ \psi)} = \vol_M(B_r)  \ge \vol_{\PP_0}(B_r) \ge B r^2, 
	$$
	contradiction. Hence, $\Sigma$ is compact. 
\end{proof}

Next, we consider rigidity results when $\Delta_{-c\eta}$ is possibly non-parabolic. To achieve the goal, we focus on the spectral properties of associated stationary Schr\"o\-ding\-er operators of the type
$$
	\Delta_{-c\eta} + a(x),
$$
for suitable potentials $a(x)$. We first recall some facts taken from \cite{BMR1} and \cite{BMR}. For $f,a \in C^0(M)$ consider the weighted Laplacian
$$
	\Delta_f = \Delta - \langle \nabla \; \cdot \; , \nabla f \rangle,
$$
that is symmetric on $C^\infty_c(M)$ if we integrate with respect to the weighted measure $e^{-f}\di x$, and let $L = \Delta_f + a(x)$. Given an open subset $\Omega \subset M$, the bottom of the spectrum of $(L, C^\infty_c(\Omega))$ is variationally defined by the Rayleigh quotient
\begin{equation}
	\lambda_1^{L}(\Omega) = \inf_{\substack{\varphi\in C^{\infty}_c(\Omega) \\ \varphi\not\equiv0}}\frac{Q_L(\varphi)}{\int_M \varphi^2 e^{-f}}, \qquad \text{with} \quad Q_L(\varphi) = \int_M\left[|\nabla\varphi|^2- a(x)\varphi^2\right]e^{-f}.
\end{equation}
We say that $L \ge 0$ on $\Omega$ if $\lambda_1^L(\Omega) \ge 0$. By a minor adaptation of a result in \cite{FCS} and \cite{MP}, $L \ge 0$ if and only if there exists a positive solution of 
$$
	\Delta_f u + a(x) u =0 \qquad \text{on } \, \Omega,
$$
and equivalently, if and only if there exists a positive solution of $\Delta_f u + a(x) u \le 0$ of class $H^1_\loc(\Omega)$.\par
The Morse index of $L$, $\ind_L(M)$, is defined as the limit 
$$
	\ind_L(M) = \lim_{j \ra \infty} \ind_L(\Omega_j), 
$$
where $\Omega_j \uparrow M$ is an exhaustion of $M$ by relatively compact, smooth open domains (the limit is independent of the exhaustion), and $\ind_L(\Omega_j)$ is the number of negative eigenvalues of $L$ on $\Omega_j$. If $L$ is bounded from below on $C^\infty_c(M)$, $\ind_L(M)$ coincides with the index of the Friedrichs extension of $L$ originally defined on $C^\infty_c(M)$, that is, the only self-adjoint extension of $L$ whose domain lies in the one of the associated quadratic form $Q_L$ (for more details, see Chapter 3 in \cite{PRS} and \cite[Sect. 1.3]{BMR}). We say that $L$ has finite index if $\ind_L(M) < \infty$. By work of \cite{D} and \cite[Thm. 1.41]{BMR}, $L$ has finite index if and only if there exists a compact set $K \subset M$ such that $L \ge 0$ on $M \backslash K$, see also \cite{FC, Pie}.\par

We are ready to prove Theorem \ref{B:thm:inv} from the Introduction.

\begin{proof}[Proof of Theorem \ref{B:thm:inv}]
	Up to possibly changing $t$ to $-t$, we consider \eqref{B:uequationgrad} for $u = \langle \nu, X \rangle \ge 0$ restricted to the end $E$, with $X=B=\partial_t$:
	\begin{equation} \label{B:uequationprod2}
		Lu = \Delta_{-c\eta}u + (|\II|^2+\Ricbar(\nu,\nu))u = 0 \quad \text{on } E.
	\end{equation}
	By the strong maximum principle, either $u \equiv 0$ or $u>0$ on $E$, and in the second case we infer $\lambda_1^L(E) \ge 0$. We shall prove that $\lambda_1^L(E) <0$, so $u\equiv0$ on $E$ and, by unique continuation, $u \equiv 0$ on the entire $M$. The conclusion of the Theorem follows from Lemma \ref{lem_splitting}, taking into account that $X$ is nowhere vanishing.
	To estimate $\lambda_1^L(E)$, let $R>0$ be large enough so that $\partial E \subseteq B_R(o)$. This is possible because $\partial E \subseteq \partial K$ is compact and $M$ is complete. We let $v_E(r), A_E(r)$ be defined as in \eqref{def_VE}, \eqref{B:ameansdef}, with $\eta = \etabar\circ\psi$ and $\etabar$, $a(x)$ as in \eqref{baretaax}.
	We consider the Cauchy problem
	\begin{equation} \label{B:PC}
		\begin{cases}
			(v_Ez')' + A_Ev_Ez = 0 & \text{on } \, [R,\infty) \\
			z(R) = 1 > 0, \quad v_E(R)z'(R^+) = 0.
		\end{cases}
	\end{equation}
	We observe that $v_E\in L^{\infty}_{\loc}([R, \infty))$ and, by Proposition 1.4 of \cite{BMR}, we also have $v_E^{-1} \in L^{\infty}_{\loc}([R, \infty))$. Hence (\ref{B:PC}) has a solution $z \in \Lip_\loc([R, \infty))$. As such, by Proposition 4.6 of \cite{BMR}, $z$ has, if any, only isolated zeros. Under the assumptions in $(i)$, using the co-area formula we have
	\begin{equation} \label{B:intAv}
		\int_{R}^r A_E(s)v_E(s)\di s = \int_{E\cap(B_r\setminus \overline{B}_{R})}\left(\Ricbar(\nu,\nu)+|\II|^2\right)e^{c\eta} \to +\infty
	\end{equation}
	as $r\to \infty$. From \eqref{B:intAv} and $v_E^{-1} \not \in L^1(\infty)$, we deduce that $z$ is oscillatory by \cite[Cor. 2.9]{MMR}. Similarly, $z$ is oscillatory under assumption $(ii)$, by \cite[Thm. 5.6]{BMR}. Choose any two consecutive zeros $R<R_1<R_2$ of $z$, and define $\varphi(x)=z(r(x))\in\Lip_{\loc}(M)$. Since $\partial E \subseteq B_R$, we have  $\partial(E \cap (B_{R_2}\setminus\overline{B}_{R_1})) = (E\cap\partial B_{R_2}) \cup (E\cap\partial B_{R_1})$ and by the Rayleigh variational characterization and the co-area formula we have
	\begin{align*}
		\lambda_1^L(E \cap (B_{R_2}\setminus\overline{B}_{R_1}) ) & \leq \frac{\int_{E \cap (B_{R_2}\setminus\overline{B}_{R_1})}e^{c\eta}\left[|\nabla\varphi|^2-\left(\Ricbar(\nu,\nu)+|\II|^2\right)\varphi^2\right]}{\int_{E \cap (B_{R_2}\setminus\overline{B}_{R_1})}e^{c\eta}\varphi^2} \\
		& = \frac{\int_{R_1}^{R_2}\left[|z'(s)|^2v_E(s)-A_E(s)v_E(s)z(s)^2\right] \di s}{\int_{R_1}^{R_2}z(s)^2v_E(s) \, \di s} = 0
	\end{align*}
	as immediately seen by multiplying the equation in (\ref{B:PC}) by $z$, integrating by parts and using $z(R_1)=z(R_2)=0$. As a result, by the monotonicity property of eigenvalues,
	\begin{equation} \label{B:1eigenvalue}
		\lambda_1^L(E) < \lambda_1^L(E \cap ( B_{R_2}\setminus\overline{B}_{R_1} ) ) \leq 0,
	\end{equation}
	as required.
\end{proof}

With a similar proof, we also deduce the following

\begin{theorem} \label{B:thm:euclidtrans}
	Let $\psi:M^m\to\RR^{m+1}$ be an oriented, connected, complete translational soliton in direction $X \in \sphere^m$ with soliton constant $c$. Suppose that there exists some end $E \subset M$ such that the image of $E$ via the Gauss map is contained in a closed hemisphere of $\sphere^m$ with center $b \in \sphere^m$. If 
	\begin{equation}\label{inte_pertranslation}
		\left( \int_{\partial B_r \cap E} e^{c\langle \psi, X\rangle} \right)^{-1} \not \in L^1(\infty), \qquad \lim_{r\to+\infty}\int_{E \cap B_r}e^{c \langle \psi, X\rangle}|\II|^2 =+\infty
	\end{equation}
	for some fixed origin $o \in M$, then $\psi$ splits as a Riemannian product $\RR \times \Sigma^{m-1}$ and, under this isometry, $\psi: \RR \times \Sigma \ra \RR^{m+1}$ can be written as the cylinder
	\begin{equation}\label{splitting2}
		\psi(t,y) = \psi(0,y) + tb.
	\end{equation}
	In particular, $\Sigma$ is a mean curvature flow soliton in the hyperplane $b^\perp$ with respect to $X - \langle X, b \rangle b$ with constant $c$ (and $\Sigma$ is minimal if $X=b$).
\end{theorem}

\begin{proof}
	The proof is identical, using as $B$ the parallel translation of $b$ and noting that $[X, B] =0$. Lemma \ref{lem_splitting} guarantees the splitting in the direction of $B$, and the fact that $\Sigma$ is a soliton in the hyperplane $b^\perp$ is a straightforward consequence. 
\end{proof}

\begin{theorem} \label{B:thm:IIbound}
	Let $\psi:M^m\to\RR^{m+1}$ be an oriented, connected, complete mean curvature flow soliton with respect to the position vector field of $\RR^{m+1}$ with soliton constant $c$. Assume that, for some fixed origin $o \in M$,
	\begin{equation} \label{B:volgrowthsphere}
		\liminf_{r \ra \infty} \frac{ \log \int_{B_r} e^{c\frac{|\psi|^2}{2}}}{r} = \alpha \in [0, \infty),
	\end{equation}
	where $B_r$ is the geodesic ball of $M$ centered at $o$. Suppose that $\langle B,\nu\rangle$ has constant sign for a chosen unit normal $\nu$ to $M$ and a parallel unit field $B$. If 
	$$
		\alpha < 2\inf_M|\II|,
	$$
	then, $M$ splits as a cylinder $\RR \times \Sigma$, the immersion writes as $\psi(t,y) = \psi(0,y) + tB$ and the restriction $\psi_\Sigma : \Sigma \ra B^\perp = \RR^m$ is itself a soliton with respect to the position field of $\RR^m$ with constant $c$.
\end{theorem}

\begin{proof}
	As we saw in the proof of Theorem \ref{B:thm:euclid}, (\ref{B:uequationgrad}) with $u=\langle B,\nu\rangle$ and $\eta=\frac{|\psi|^2}{2}$ becomes
	\begin{equation} \label{B:uequationeuclid}
		\Delta_{-c\eta}u+|\II|^2u=0.
	\end{equation}
	Supposing, without loss of generality, that $u\geq0$, by the usual maximum principle and equation (\ref{B:uequationeuclid}) we have that either $u\equiv0$ or $u>0$ on $M$. In the latter case, using a minor variation of a result of Barta, \cite{B}, we have
	$$		
		\lambda_1^{\Delta_{-c\eta}}(M) \geq \inf_M\left(-\frac{\Delta_{-c\eta}u}{u}\right) = \inf_M |\II|^2,
	$$
	where the last equality is due to (\ref{B:uequationeuclid}). However, by adapting to the weighted setting a criterion due to R. Brooks and Y. Higuchi \cite{brooks, higuchi} (cf. \cite{rocha}), if \eqref{B:volgrowthsphere} holds then
	\begin{equation}\label{eq_higu}
		\lambda_1^{\Delta_{-c\eta}}(M) \le \frac{\alpha^2}{4}.
	\end{equation}
	Coupling the last two inequalities we contradict our assumption $\alpha < 2 \inf_M|\II|$. Therefore, $u \equiv 0$ and, by Lemma \ref{lem_splitting}, $M$ splits as indicated. The fact that $\Sigma$ is a soliton in $B^\perp$ is easily established. 
\end{proof}

\begin{remark}
	The criterion in \cite{brooks, higuchi, rocha} is, indeed, an estimate on the bottom of the essential spectrum of $\Delta_{-c\eta}$ and also requires that the weighted volume $\vol_{-c\eta}(M)$ be infinite.  However, if $\vol_{-c\eta}(M) < \infty$ then clearly the bottom of the spectrum of $\Delta_{-c\eta}$ is zero, as constant functions are in $L^2$, and \eqref{eq_higu} still holds.
\end{remark}

\section{Stability}

The aim of this section is to study stable mean curvature flow solitons. Again we refer to Section 11 in \cite{AdeLR} and for the sake of simplicity we restrict ourselves to the case of codimension $1$ mean curvature flow solitons in a warped product space $\Mbar^{m+1}=I\times_h\PP$. For a given immersion $\psi:M^m\to\Mbar^{m+1}$ we let $\eta=\etabar\circ\psi$ with $\etabar$ defined in (\ref{B:defetabar}). For a fixed $c\in\RR$ and a relatively compact $\Omega\subseteq M$ we consider the weighted volume
\begin{equation} \label{stab:defwvol}
	\vol_{c\eta}(\Omega) = \int_{\Omega}e^{c\eta} \di x
\end{equation}
where $\di x$ is the volume element induced by the immersion $\psi$. The mean curvature flow soliton equation
$$
	m\Hvec=cX^{\bot}
$$
with $c$ the soliton constant, is the Euler-Lagrange equation for the functional (\ref{stab:defwvol}) with respect to variations of $\psi$ compactly supported in $\Omega$. See, for instance, Proposition 11.1 in \cite{AdeLR}. In the same Proposition we also compute the stability operator $L$ given by
\begin{equation} \label{stab:stoperator}
	L= \Delta_{-c\eta} + \Big( (|\II|^2+\Ricbar_{-c\eta}(\nu,\nu)\Big) = \Delta_{-c\eta}+\Big(|\II|^2+\Ricbar(\nu,\nu)-ch'(\pi_I\circ\psi)\Big)
\end{equation}
where $\nu$ is any (local) unit normal vector field to the stationary immersion $\psi$. It follows that the LHS of equation (\ref{B:uequationgrad}) is exactly the stability operator applied to the function $u$.

We say that $\psi$ is stable if $L \ge 0$ on $M$. We are interested in stable solitons in warped product ambient manifolds $\Mbar = \RR \times_h \PP$ with constant sectional curvature $\bar\kappa$. In this case, necessarily $\PP$ has constant sectional curvature too, say $\kappa$, and by Gauss equations
\begin{equation} \label{stab:constcurv}
	-\frac{h''}{h} = \bar \kappa = \frac{\kappa}{h^2} - \left(\frac{h'}{h}\right)^2,
\end{equation}
thus 
\begin{equation}\label{stab:gausseq}
	\kappa + h''h - (h')^2 \equiv 0.
\end{equation}

We shall use the next result in \cite[Thm. 4]{BPR}, whose proof is a refinement of that of Theorem 4.5 in \cite{PRS}.

\begin{theorem} \label{stab:thm:PRS}
	Let $(M,\metric,e^{-f})$ be a connected, complete weighted manifold with $f\in C^{\infty}(M)$. Let $a(x)\in L^{\infty}_{\loc}(M)$ and let $u\in\Lip_{\loc}(M)$ satisfy the differential inequality
	\begin{equation}
		u\Delta_f u + a(x)u^2 + A|\nabla u|^2 \geq 0 \quad \text{weakly on } M
	\end{equation}
	for some constant $A\in\RR$. Let $v\in\Lip_{\loc}(M)$ be a positive solution of
	\begin{equation} \label{stab:bigthmveq}
		\Delta_f v + \mu a(x)v \leq -K\frac{|\nabla v|^2}{v} \quad \text{weakly on } M
	\end{equation}
	for some constants $\mu$ and $K$ satisfying
	\begin{equation} \label{stab:bigthmparam}
		K>-1, \quad A+1\leq\mu(K+1), \quad \mu>0, \quad \mu\geq A+1.
	\end{equation}
	Suppose that one of the following conditions is satisfied:
	\begin{enumerate}
		\item $M$ is compact, or
		\item $M$ is non-compact and there exists $o \in M$ such that
		\begin{equation} \label{stab:bigthmint}
		\left(\int_{\partial B_r}v^{\frac{\beta+1}{\mu}(2-p)}|u|^{p(\beta+1)}e^{-f}\right)^{-1} \not\in L^1(+\infty)
		\end{equation}
		holds for some $p>1$ and $\beta$ satisfying
		\begin{equation}
		A \leq \beta \leq \mu(K+1)-1, \quad \beta > -1,
		\end{equation}
		with $B_r$ the geodesic ball of $M$ centered at $o$ with radius $r$.
	\end{enumerate}
	Then there exists a constant $\alpha\in\RR$ such that
	\begin{equation}
		\alpha v = \sgn u |u|^{\mu}.
	\end{equation}
	Furthermore,
	\begin{itemize}
		\item [(i)] if $A+1<\mu(K+1)$ then $u$ is constant on $M$ and if in addition $a(x)$ does not vanish identically, then $u\equiv0$;
		\item [(ii)] if $A+1=\mu(K+1)$ and $u$ doesn't vanish identically, then $v$ and therefore $|u|^{\mu}$ satisfy (\ref{stab:bigthmveq}) with the equality sign.
	\end{itemize}
\end{theorem}

\begin{remark}
	In the references given above, Theorem \ref{stab:thm:PRS} is proved in the non-compact case. However, in the compact case the proof simplifies, as no cut-off functions are needed.
\end{remark}

We are now ready to prove the next

\begin{theorem} \label{stab:thm:geod}
	Let $\Mbar=I\times_h\PP^m$ be a warped product space with constant sectional curvature $\bar\kappa$. Let $\psi:M^m\to\Mbar$ be a connected, complete, stable mean curvature flow soliton with respect to $X=h(t)\partial_t$ with soliton constant $c$. Assume that 
	\begin{equation}\label{stab:assu_cmkappa_H}
		\inf_M ch'(\pi_I\circ\psi)  < \sup_M ch'(\pi_I\circ\psi)\leq \frac{1}{2} m\bar\kappa 
	\end{equation}
	and $|\Hvec|\in L^2(M,e^{c\eta})$. Then $\Hvec\equiv0$.
\end{theorem}

\begin{proof}
	Since the stability operator $L$ is non-negative, there exists a positive smooth solution $v$ of the equation
	\begin{equation} \label{stab:positivesolution}
		L v = \Delta_{-c\eta}v + (|\II|^2+\Ricbar_{-c\eta}(\nu,\nu))v = 0 \quad \text{on } M
	\end{equation}
	where $\nu$ is any (local) normal to $M$. We note that
	$$
		\Ricbar_{-c\eta}(\nu,\nu) = \Ricbar(\nu,\nu) - ch'(\pi_I\circ\psi)\langle\nu,\nu\rangle_{\Mbar} = m\bar\kappa - ch'(\pi_I\circ\psi)
	$$
	Using that $\Mbar$ and $\PP$ have constant sectional curvature, respectively $\bar \kappa$ and $\kappa$, from Corollary 7.3 in \cite{AdeLR} one has the validity of the formula
	\begin{equation} \label{stab:Hsquareequationsimpl}
		\frac{1}{2}\Delta_{-c\eta}|\Hvec|^2 = -ch'|\Hvec|^2 - |\II_{\Hvec}|^2 + |\nabla^\perp\Hvec|^2,
	\end{equation}
	where $\II_\Hvec = \langle \II, \Hvec \rangle$. Since $M$ has codimension $1$, we respectively have
	$$
		|\II_{\Hvec}|^2 = |\Hvec|^2|\II|^2, \qquad |\nabla|\Hvec|^2|^2 = 4|\Hvec|^2|\nabla^\perp\Hvec|^2.
	$$
	Hence multiplying (\ref{stab:Hsquareequationsimpl}) by $|\Hvec|^2$ and using the above inequalities we deduce
	$$
		|\Hvec|^2\Big(\Delta_{-c\eta}|\Hvec|^2+2(|\II|^2+ch')|\Hvec|^2\Big) = \frac{1}{2}|\nabla|\Hvec|^2|^2.
	$$
	Setting $u=|\Hvec|^2$ we can rewrite the above as
	\begin{equation}\label{eq_1}
		u\Delta_{-c\eta}u + 2(|\II|^2+m\bar\kappa-ch')u^2 = \frac{1}{2}|\nabla u|^2 - 2(2ch'-m\bar\kappa)u^2,
	\end{equation}
	and because of \eqref{stab:assu_cmkappa_H} we deduce
	\begin{equation}\label{ineq_u}
		u\Delta_{-c\eta}u + 2(|\II|^2+m\bar\kappa-ch')u^2 \geq \frac{1}{2}|\nabla u|^2.
	\end{equation}
	This compares with (\ref{stab:positivesolution}) that we can rewrite in the form
	\begin{equation} \label{stab:positivesolution2}
		\Delta_{-c\eta}v + (|\II|^2+m\bar\kappa-ch')v = 0.
	\end{equation}
	We set
	$$
		a(x) = 2(|\II|^2+m\bar\kappa-ch')(x) \quad \text{on } M
	$$
	and let $\mu=\frac{1}{2}$, $K=0$, $A=-\frac{1}{2}$, $f=-c\eta$. If $M$ is compact, we can apply Theorem \ref{stab:thm:PRS} above. If $M$ is non-compact, we set $p=2$, $\beta=-\frac{1}{2}$ and, since in this case condition (\ref{stab:bigthmint}) reads as
	$$
		\left(\int_{\partial B_r}|u|e^{c\eta}\right)^{-1}\not\in L^1(+\infty),
	$$
	which is implied by $|\Hvec|^2\in L^1(M,e^{c\eta})$, so we can still apply Theorem \ref{stab:thm:PRS} to deduce that either $u\equiv0$, that is, $\Hvec\equiv0$, or $u^{\frac{1}{2}}$ satisfies (\ref{stab:positivesolution2}), that is,
	$$
		\Delta_{-c\eta}u^{\frac{1}{2}} + (|\II|^2+m\bar\kappa-ch')u^{\frac{1}{2}} = 0.
	$$
	Explicitating the equation with respect to $u$ we have
	$$
		u\Delta_{-c\eta}u + 2(|\II|^2+m\bar\kappa-ch')u - \frac{1}{2}|\nabla u|^2 = 0.
	$$
	By \eqref{eq_1}, we infer $2ch' - m\bar\kappa \equiv 0$ on $M$, contradicting \eqref{stab:assu_cmkappa_H}.
\end{proof}

By similar reasoning, but exploiting the Jacobi equation instead of \eqref{stab:Hsquareequationsimpl}, we recover the following well-known result for self-shrinkers in Euclidean space (cf. Theorem 9.2 in \cite{CM_generic}).

\begin{proposition} \label{stab:prp:euclid}
	There are no stable, complete oriented self-shrinkers $\psi:M^m\to\RR^{m+1}$ (with soliton constant $c<0$) satisfying
	\begin{equation} \label{stab:euclidint}
		e^{\frac{c}{2}|\psi|^2} \in L^1(M).
	\end{equation}
\end{proposition}


\begin{proof}
	It is enough to observe that, for each unit parallel field $B$, by \eqref{neg_firsteigen} $u=\langle B,\nu\rangle$ satisfies
	\begin{equation}\label{uno}
	\Delta_{-c\frac{|\psi|^2}{2}}u + (|\II|^2-c)u = -cu.
	\end{equation}
	Because of \eqref{stab:euclidint} and $|u| \le 1$, $u \in L^2(M, e^{c\eta})$ and is therefore an eigenfunction associated to the eigenvalue $c<0$. Alternatively, we can directly use Theorem \ref{stab:thm:PRS} with the choices
	\[
	a(x) = (|\II|^2-c)(x), \quad A=0, \quad \mu=1, \quad K=0,
	\]
	and
	\[
	\beta=0, \quad p=2 \quad \text{in case $M$ is non-compact,}
	\]
	noting that stability implies the existence of a positive smooth solution $v$ of the equation
	\begin{equation} \label{stab:positivesolution3}
	\Delta_{-c\frac{|\psi|^2}{2}}v + (|\II|^2-c)v = 0 \quad \text{on } M,
	\end{equation}
	that \eqref{uno} implies 
	\[
	u\Delta_{-c\frac{|\psi|^2}{2}}u + (|\II|^2-c)u^2 = -cu^2 \ge 0,
	\]
	and that in the non-compact case \eqref{stab:bigthmint} is satisfied because of \eqref{stab:euclidint}. Applying (ii) of Theorem \ref{stab:thm:PRS} to each connected component of $M$ we have that either $u\equiv0$ and $\psi(M)$ is orthogonal to $B$ or $u=\sgn u |u| \not\equiv 0$ and satisfies \eqref{stab:positivesolution3}. In the second case, we get $cu^2=0$, a contradiction since $c<0$. Thus $\psi(M)$ is orthogonal to $b$, an evident contradiction since $b$ is chosen arbitrarily. 
\end{proof}

We next come to the proof of Theorem \ref{teo_umbilico_intro} in the Introduction. To get more information besides stability and equation \eqref{stab:Hsquareequationsimpl}, we shall also use a Simons' type formula for mean curvature flow solitons (cf. Lemma 10.8 in \cite{CM_generic} for self-shrinkers in $\RR^{m+1}$, and \cite{CMZ} for general $f$-minimal hypersurfaces). We first introduce the following

\begin{lemma}\label{lem_kato}
	Let $A$ be a $2$-covariant Codazzi tensor field on a connected manifold $(M^m, \metric)$, that is, $A$ satisfies
	$$
	(\nabla_Z A)(X,Y) = (\nabla_Y A)(X,Z) \qquad \text{for each } \, X,Y,Z \in TM.
	$$
	Suppose that its traceless part $B = A - \frac{\mathrm{tr}A}{m} \metric$ matches the identity in Kato's inequality: 
	\begin{equation}\label{ide_kato}
	|\nabla |B||^2 = |\nabla B|^2.
	\end{equation}
	Then, either 
	\begin{itemize}
		\item[(i)] $B \equiv 0$ on $M$, or
		\item[(ii)] $|B|>0$ and $\nabla B = \nabla A = 0$ on $M$, or
		\item[(iii)] the subset $M_0 = \{ p \in M : |B|(p) \neq 0, \nabla|B|(p) \neq 0 \}$ is nonempty and each point $p\in M_0$ has a neighbourhood $U$ that splits as a Riemannian product $(-\varepsilon,\varepsilon) \times \Sigma$. Furthermore, $\metric$ and $A$ write as
		$$
		\metric = \di s \otimes \di s + \metric_{\Sigma}, \qquad	A = \mu_1(s) \, \di s \otimes \di s + \mu_2(\pi_{\Sigma}) \metric_{\Sigma},
		$$
		with $\pi_{\Sigma} : U \to \Sigma$ the projection onto $\Sigma$ and $\mu_1(s) \neq \mu_2(\pi_{\Sigma})$ for every point of $U$. If $m \geq 3$, then $\mu_2$ is constant on $\Sigma$.
	\end{itemize}
\end{lemma}

\begin{proof}
	For notational convenience let $H = \mathrm{tr}A/m$. Suppose that $B \not \equiv 0$, and let $p$ be such that $|B|(p)>0$. Let let $a_{ij}, b_{ij}$ be the components of $A,B$ in a local orthonormal frame $\{e_i\}$, with $b_{ij} = a_{ij} - H \delta_{ij}$, and assume that $b$ is diagonalized at $p$ with eigenvalues $\{\lambda_i\}$. At the point $p$ we then have
	\begin{equation}\label{primaalg}
	\begin{array}{lcl}
	4|B|^2 |\nabla |B||^2 & = & |\nabla |B|^2|^2 = \disp 4 \sum_k \left( \sum_{i,j} b_{ij}b_{ijk} \right)^2 = 4 \sum_k \left( \sum_i \lambda_i b_{iik} \right)^2 \\[0.5cm]
	& = & \disp 4 \sum_k \left[ \sum_i \lambda_i^2 \sum_i b_{iik}^2 - \sum_{j \neq i} (\lambda_i b_{jjk}- \lambda_j b_{iik})^2\right] \\[0.5cm]
	& = & \disp 4|B|^2 \left[ \sum_{i,k}b_{iik}^2 - |B|^{-2} \sum_k\sum_{j \neq i} (\lambda_i b_{jjk}- \lambda_j b_{iik})^2\right].
	\end{array}
	\end{equation}
	(recall that $u = |B|^2>0$ around $p$). Therefore, at $p$
	$$
	\begin{array}{lcl}
	0 & = & \disp |\nabla B|^2 - |\nabla |B||^2 =  \sum_{i,j,k} b_{ijk}^2 - \sum_{i,k}b_{iik}^2 + |B|^{-2} \sum_k\sum_{j \neq i} (\lambda_i b_{jjk}- \lambda_j b_{iik})^2 \\[0.5cm]
	& = & \disp \sum_{i,j,k \ \text{distinct}} b_{ijk}^2 + \sum_{i \neq k} (b_{iki}^2 + b_{ikk}^2) + |B|^{-2} \sum_k\sum_{j \neq i} (\lambda_i b_{jjk}- \lambda_j b_{iik})^2.
	\end{array}
	$$
	Thus,  
	\begin{equation}\label{rigid}
	\begin{array}{lll}
	(i) & 0 = b_{ijk} = b_{iki} = b_{ikk} & \quad \text{for each  } \, i,j,k \, \text{ pairwise distinct.} \\[0.2cm]
	(ii) & \lambda_i b_{jjk}- \lambda_j b_{iik} = 0 & \quad \text{for each $k$ and each $i \neq j$.}
	\end{array}
	\end{equation}
	Since $A$ is a Codazzi tensor, $a_{ijk}= a_{ikj}$ for each $i,j,k$ and \eqref{rigid} implies
	$$
	0 = a_{ijk} = a_{iki} = a_{ikk} = a_{iik} \qquad \text{for $i,j,k$ distinct},
	$$
	and in particular from $a_{iik}=0$ we get 
	\begin{equation}\label{zeroalge}
	b_{iik} = -H_k \qquad \text{for each $k$ and each $i \neq k$.} 
	\end{equation}
	As $B$ is trace-free, 
	\begin{equation}\label{trace_b}
	b_{kkk} = - \sum_{i \neq k} b_{iik} = (m-1)H_k \qquad \text{for each } \, k. 
	\end{equation}
	
	Suppose that $\nabla |B| \equiv 0$, so by \eqref{ide_kato} $\nabla B \equiv 0$. From \eqref{trace_b} we deduce that $H$ is constant, hence also $\nabla A =0$. The constancy of the norm of $B$ implies that $B$ never vanishes on $M$, thus $(ii)$ holds.\\
	We are left to consider the case $|\nabla |B|| \neq 0$ at some point $p$. Because of \eqref{primaalg}, there exist $i,k$ such that $b_{iik} \neq 0$, and in view of \eqref{zeroalge}, \eqref{trace_b} we can assume $i=k$ and rename indices in such a way that $b_{111} = (m-1)H_1\neq 0$. Using \eqref{rigid}, for each $j \neq 1$ we infer
	$$
	-\lambda_1 H_1 = \lambda_1 b_{jj1} = \lambda_j b_{111} = (m-1)\lambda_j H_1,
	$$
	thus $\lambda_1 = -(m-1)\lambda_j$ for each $j \neq 1$. The tensor $B$ has only two eigenvalues, both nonzero since $|B|>0$ and therefore distinct. The rank theorem guarantees that the frame $\{e_i\}$ can be chosen to diagonalize $B$ in a full neighbourhood $U \subseteq M$ of $p$, hence 
	\begin{equation}\label{eq_Hj}
	\begin{array}{ll}
	- H_k = b_{jjk} = e_k(\lambda_j) \qquad \text{for } \, k \neq j, \\[0.2cm]
	(m-1)H_j = b_{jjj} = e_j(\lambda_j) 
	\end{array}
	\end{equation}
	If we denote with $\mu_j = \lambda_j + H$ the eigenvalues of $A$, from \eqref{eq_Hj} we deduce $e_k(\mu_j)=0$ for $k \neq j$, hence $\nabla \mu_j$ is a multiple of $e_j$. Next, denote with $\{\theta^i\}$ the local orthonormal coframe dual to the frame $\{e_i\}$ and by $\{\theta^i_j\}$ the corresponding Levi-Civita connection forms satisfying the structural equations
	\begin{equation}\label{str_eq}
	\begin{array}{ll}
	\di \theta^i = - \theta^i_j \wedge \theta^j & \text{for } \, 1 \leq i \leq m, \\ [0.2cm]
	\theta^i_j + \theta^j_i = 0 & \text{for } \, 1 \leq i, j \leq m;
	\end{array}
	\end{equation}
	from \eqref{rigid} we deduce, for each $k$ and each $j \neq 1$
	$$
	0 = b_{1jk} = e_k(b_{1j}) - b_{1i} \theta^i_j (e_k) - b_{ij} \theta^i_1 (e_k) = 0 - \lambda_1 \theta^1_j (e_k) - \lambda_j \theta^j_1 (e_k) = m \lambda_j \theta^1_j (e_k)
	$$
	on $U$, since $\lambda_1 = - (m - 1) \lambda_j$ and $\theta^j_1 = - \theta^1_j$ (no sum over $j$ is intended in the expression above). The eigenvalues of $B$ are nonzero, so $\theta^1_j \equiv 0$ for each $j$. It follows from the first of \eqref{str_eq} that $\di \theta^1 = 0$ and around $p$ there exists $s$ such that $\theta^1= \di s$. Again the structure equations imply $\mathcal{L}_{\nabla s} \metric = 0$ so $\nabla s$ is Killing, hence parallel. Therefore, a neighbourhood $U$ of $p$ splits as $(-\ep, \ep) \times \Sigma$ with the product metric. Eventually, since $\nabla \mu_j$ only depends on $e_j$ and $A$ is diagonalized by $\{e_j\}$ on $U$, $A$ too has the form indicated in $(iii)$, with $\mu_2$ constant if $\Sigma$ has dimension at least $2$.
\end{proof}

We are ready to prove Theorem \ref{teo_umbilico_intro}.

\begin{theorem} \label{stab:thm:umbilicity}
	Let $\psi:M^m\to\Mbar^{m+1}=I\times_h\PP$ be a connected, complete, stable mean curvature flow soliton with respect to $X=h(t)\partial_t$ with soliton constant $c$. Assume that $\Mbar$ is complete and has constant sectional curvature $\bar\kappa$, with
	\begin{equation} \label{stab:ch'bound}
	ch'(\pi_I\circ\psi)\leq m\bar\kappa \quad \text{on } \, M.
	\end{equation}
	Let $\Phi=\II-\metric_M\otimes\Hvec$ be the umbilicity tensor of $\psi$ and suppose that
	\begin{equation} \label{stab:PhiL2}
	|\Phi|\in L^2(M,e^{c\eta})
	\end{equation}
	with $\eta$ defined as above. Then one of the following cases occurs:
	\begin{enumerate}
		\item[(i)] $\psi$ is totally geodesic (and if $c \neq 0$ then $\psi(M)$ is invariant by the flow of $X$), or
		\item[(ii)] $I = \RR$, $h$ is constant on $\RR$, $\overline{M}$ is isometric to the product $\RR \times F$ with $F$ a complete flat manifold and $M$ is also flat. By introducing the universal coverings $\pi_M : \RR^m \to M$, $\pi_F : \RR^m \to F$ and $\pi_{\Mbar} = \mathrm{id}_{\RR} \times \pi_F : \RR^{m+1} \to \Mbar$, the map $\psi$ lifts to an immersion $\hat \psi : \RR^m \to \RR \times \RR^m$ satisfying $\pi_{\Mbar} \circ \hat\psi = \psi \circ \pi_M$, which up to an isometry of $\RR^m$ and a translation along the $\RR$ factor of $\RR^{m+1}$ is given by
		$$
		\hat\psi : \RR^m \to \RR^{m+1}, \quad (x^1,x^2,\dots,x^m) \mapsto (\sigma_1(x^1),\sigma_2(x^1),x^2,\dots,x^m)
		$$
		where $\gamma = (\sigma_1,\sigma_2) : \RR \to \RR^2$ is the grim reaper curve with image
		$$
		\sigma(\RR) = \left\{ (x,y) \in \RR^2 : x = - \frac{1}{ch_0} \log( \cos (ch_0 y)), |y| < \frac{2}{\pi |c| h_0} \right\}
		$$
		and $h_0$ is the constant value of $h$ on $\RR$. Furthermore, there exists a Riemannian submersion $\pi_\Omega : M \to \Omega$ onto a compact, flat manifold $\Omega$ with $1$-dimensional, noncompact geodesic fibers of the type $\pi_M(\RR \times \{ (x^2,\dots,x^m) \} )$, for constant $(x^2,\dots,x^m) \in \RR^{m-1}$. Such fiber is mapped by $\psi$ into the grim reaper curve $\pi_{\Mbar}(\sigma(\RR) \times \{ (x^2,\dots,x^m) \})$.
	\end{enumerate}
Furthermore, any of the solitons in (ii) is stable, while a soliton in (i) is stable if and only if $L = \Delta_{-c\eta} + (m\bar\kappa-ch')$ is non-negative.
\end{theorem}

\begin{remark}\label{rem_nice}
		The completeness assumption on $\overline{M}= I \times_h \PP$ forces $I = \RR$. However, completeness can be weakened to the combination of the following two requirements:
		\begin{itemize}
			\item[-] $\PP$ is complete;
			\item[-] either $I = \RR$ or $h$ extends continuously to $\partial I \subset \RR$ with value $0$. 
		\end{itemize}
	\end{remark}

\begin{proof}
We split the rather long proof into several steps. \\[0.3cm] 
\textbf{Step 1:} the function $u = |\Phi|^2$ either vanishes identically, or it is everywhere positive and satisfies
	\begin{equation} \label{stab:Phieq3}
	u\Delta_{-c\eta}u+2(|\II|^2+m\bar\kappa-ch'(\pi_I\circ\psi))u^2=\frac{1}{2}|\nabla u|^2.
	\end{equation}

\noindent \emph{Proof of Step 1.}\\
Having fixed a local unit normal $\nu$ and set $H=\langle\Hvec,\nu\rangle$ we clearly have
	\[
	|\Phi|^2=|\II|^2-mH^2\geq0.
	\]
	Furthermore,
	\begin{equation} \label{stab:nablaPhi}
	|\nabla\Phi|^2=|\nabla \II|^2-m|\nabla H|^2.
	\end{equation}
	We recall, see for instance equation (9.37) in \cite{AdeLR}, that in the present setting we have the validity of the Simons' type formula
	\begin{equation} \label{stab:|A|eq}
	\frac{1}{2}\Delta_{-c\eta}|\II|^2=-(ch'(\pi_I\circ\psi)+|\II|^2)|\II|^2+m\bar\kappa|\Phi|^2+|\nabla \II|^2.
	\end{equation}
	On the other hand, from equation (\ref{stab:Hsquareequationsimpl}) valid in the present assumptions, we have
	\begin{equation} \label{stab:Heq}
	\frac{1}{2}\Delta_{-c\eta}H^2=-(ch'(\pi_I\circ\psi)+|\II|^2)H^2+|\nabla H|^2.
	\end{equation}
	Putting together (\ref{stab:|A|eq}) and (\ref{stab:Heq}) and using (\ref{stab:nablaPhi}) we obtain
	\begin{equation} \label{stab:Phieq}
	\frac{1}{2}\Delta_{-c\eta}|\Phi|^2+(ch'(\pi_I\circ\psi)+|\II|^2-m\bar\kappa)|\Phi|^2-|\nabla\Phi|^2=0.
	\end{equation}
	The function $u=|\Phi|^2$ therefore solves
	\begin{equation} \label{stab:Phieq2}
	u\Delta_{-c\eta}u+2(|\II|^2+m\bar\kappa-ch'(\pi_I\circ\psi))u^2=2|\nabla\Phi|^2u+4(m\bar\kappa-ch'(\pi_I\circ\psi))u^2.
	\end{equation}
	Now by  (\ref{stab:ch'bound}) we deduce
	\[
	m\bar\kappa - ch' \geq 0 \quad \text{on } M,
	\]
	while from Kato's inequality $|\nabla |\Phi||^2 \le |\nabla \Phi|^2$ we get
	\begin{equation}\label{kato}
	2u|\nabla\Phi|^2 \geq \frac{1}{2}|\nabla u|^2.
	\end{equation}
	Substituting in the above we eventually have
	\begin{equation}
	u\Delta_{-c\eta}u+2(|\II|^2+m\bar\kappa-ch'(\pi_I\circ\psi))u^2 \geq \frac{1}{2}|\nabla u|^2.
	\end{equation}
	The stability of the soliton implies the validity of (\ref{stab:positivesolution2}) for some positive smooth function $v$ on $M$. We now apply Theorem \ref{stab:thm:PRS} with the choices
	\[
	a(x) = 2(|\II|^2+m\bar\kappa-ch'(\pi_I\circ\psi))(x), \quad f=-c\eta, \quad \mu=\frac{1}{2}, \quad A=-\frac{1}{2}, \quad K=0
	\]
	so that the requirements in formula (\ref{stab:bigthmparam}) are satisfied. In case $M$ is non-compact, by choosing the admissible $\beta=-\frac{1}{2}$ we see that (\ref{stab:PhiL2}) implies
	\[
	\left(\int_{\partial B_r}u e^{c\eta}\right)^{-1}\not\in L^1(+\infty),
	\]
	that corresponds to (\ref{stab:bigthmint}) for the choice $p=2$. Applying Theorem \ref{stab:thm:PRS} we deduce that either $u\equiv0$ (so $\psi:M\to\Mbar$ is totally umbilical) or $u>0$ and $u^{\frac{1}{2}}$ satisfies the equation
	\[
	\Delta_{-c\eta}u^{1/2}+(|\II|^2+m\bar\kappa-ch'(\pi_I\circ\psi))u^{1/2}=0
	\]
	or equivalently
	\begin{equation} \label{stab:Phieq3}
	u\Delta_{-c\eta}u+2(|\II|^2+m\bar\kappa-ch'(\pi_I\circ\psi))u^2=\frac{1}{2}|\nabla u|^2.
	\end{equation}
This proves Step 1. \\[0.3cm]
\textbf{Step 2:} if $u \equiv 0$, then $\psi$ is totally geodesic and, if $c \neq 0$, $X$ is tangent to $\psi(M)$. The stability of $\psi$ is equivalent to the non-negativity of $L = \Delta_{-c\eta} + (m\bar\kappa-ch')$. \\[0.2cm]
\noindent \emph{Proof of Step 2.}\\
$\Mbar$ is a space of constant curvature, so the tensor field $\II$ is Codazzi. Fixing a local orthonormal coframe $\{\theta^i\}_i$ on $M$ we write $\II = a_{ij} \theta^i \otimes \theta^j \otimes \nu$, $\nabla \II = a_{ijk} \theta^k \otimes \theta^i \otimes \theta^j \otimes \nu$, $dH = H_k \theta^k$. Since $u \equiv 0$, the umbilicity tensor $\Phi = \II - \metric \otimes \Hvec$ vanishes, so we have $a_{ij} = H \delta_{ij}$ and $a_{ijk} = \delta_{ij} H_k$ by parallelism of the metric. For $1 \leq k \leq m$ and for any index $t \neq k$ we have $H_k = \delta_{tt} H_k = a_{ttk} = a_{tkt} = \delta_{tk}H_t = 0$, where $a_{ttk} = a_{tkt}$ holds true as $\II$ is Codazzi. It follows that $dH \equiv 0$, therefore $H$ is constant and so is $|\II|^2 = mH^2$. Plugging this into \eqref{stab:Heq} we get
	$$
	H^2 (ch'(\pi_I \circ \psi) + mH^2) \equiv 0 \quad \text{on } \, M.
	$$
	In particular, either $H \equiv 0$, and $\psi$ is totally geodesic, or $H \neq 0$, $ch'(\pi_I \circ \psi) \neq 0$, $ch'(\pi_I \circ \psi) = - mH^2$ on $M$. We now prove that the second case cannot occur.
	
	Suppose, by contradiction, that $u \equiv 0$, $H \neq 0$ and that $\pi_I \circ \psi$ is constant: then $\psi(M)$ is a slice $\{t_0\} \times \PP$ for some $t_0 \in I$ such that $ch'(t_0) = - mH^2 = - m \frac{h'(t_0)^2}{h(t_0)^2}$ and the stability operator of $\psi$ is
	\begin{equation}\label{eq_stab_op_pf}
	\begin{array}{lcl}
	\disp \Delta_{-c\eta} + ( |\II|^2 + m \bar\kappa - ch'(t_0) ) & = & \disp \Delta + m \left( H^2 + \frac{\kappa - h'(t_0)^2}{h(t_0)^2} + H^2 \right) \\[0.4cm]
	& = & \disp \Delta + m \, \frac{\kappa + h'(t_0)^2}{h(t_0)^2},
	\end{array}
	\end{equation}
	where $\Delta$ is the Laplace-Beltrami operator of $(\PP, h(t_0)^2 \metric_{\PP})$ and we have used the identity \eqref{stab:constcurv}. Condition \eqref{stab:ch'bound} now reads as
	$$
	c h'(t_0) = - m \, \frac{h'(t_0)^2}{h(t_0)^2} \leq m \bar\kappa = m \, \frac{\kappa}{h(t_0)^2} - m \, \frac{h'(t_0)^2}{h(t_0)^2},
	$$
	that is $\kappa \geq 0$. Since $(\PP, h(t_0)^2 \metric_{\PP})$ has constant sectional curvature $\frac{\kappa}{h(t_0)^2} \ge 0$, the bottom of the spectrum of $-\Delta$ is zero (cf. \cite{wang}). It follows that the stability operator \eqref{eq_stab_op_pf} is non-negative if and only if $\kappa = 0$ and $h'(t_0) = 0$. But $h'(t_0) = 0$ implies that $H^2 = \frac{h'(t_0)^2}{h(t_0)^2} = 0$, contradicting the assumption that $H \neq 0$.
	
	Suppose, by contradiction again, that $u \equiv 0$, $H \neq 0$ and that $\pi_I \circ \psi$ is not constant on $M$. Then $h'$ is constant on the nondegenerate interval $(\pi_I \circ \psi)(M) \subseteq I$ and by \eqref{stab:constcurv} this forces $\bar\kappa = 0$. From $|\II|^2 = mH^2$ and $ch'(\pi_I \circ \psi) = - mH^2$, the stability operator becomes 
	\[
		L = \disp \Delta_{-c\eta} + ( |\II|^2 + m \bar\kappa - ch'(\pi_I \circ \psi) ) = \disp \Delta_{-c\eta} + 2 mH^2.
	\]	
	The Gauss equation and the fact that $\psi$ is totally umbilic imply that $M$ has constant sectional curvature $H^2>0$, whence it is compact. Therefore, the first eigenvalue of $L$	is $-2mH^2<0$, contradiction. So far, we have proved that if $u \equiv 0$ then $\psi$ must be totally geodesic. The stability operator reduces to $\Delta_{-c\eta} + (m\bar\kappa-ch')$. If $c \neq 0$, the fact that $X$ is tangent to $\psi(M)$ follows by the soliton equation, which concludes the proof of Step 2.\\[0.3cm]
%
\textbf{Step 3:} if $u>0$ solves \eqref{stab:Phieq3}, then $I = \RR$, $\Mbar$ is flat, $X = h_0 \partial_t$ for some constant $h_0>0$, and $u$ is not constant on any open subset of $M$. \\[0.2cm] 
\noindent \emph{Proof of Step 3.} From \eqref{stab:Phieq2} and \eqref{stab:Phieq3} we have
	\[
	\frac{1}{2}|\nabla|\Phi|^2|^2 = 2u|\nabla\Phi|^2 + 4(m\bar\kappa-ch')u^2.
	\]
	Since $m\bar\kappa-ch'\geq0$ and $u>0$, from the above and Kato's inequality \eqref{kato} we deduce
	\begin{equation}\label{condisplit}
	m\bar\kappa\equiv ch'(\pi_I\circ\psi), \qquad |\nabla \Phi|^2 = |\nabla |\Phi||^2 \quad \text{on } M.
	\end{equation}
Note that $u$ cannot be constant on an open set of $M$, since otherwise equation \eqref{stab:Phieq3} would reduce to $0 = 2 |\II|^2 u^2$, which is absurd since $|\II|^2 \geq u > 0$.

	We prove that $\Mbar$ is flat. Indeed, if $c = 0$ then $\bar\kappa = 0$ by \eqref{condisplit}; if $c \neq 0$, then again by \eqref{condisplit} and since $\psi(M)$ is not a slice (as slices are totally umbilical), we see that $h'$ is constant on the nondegenerate interval $(\pi_I \circ \psi)(M) \subseteq I$, so $\bar\kappa = - h''/h \equiv 0$ by \eqref{stab:constcurv}. Inserting this into \eqref{condisplit} we see that $h' \equiv 0$, so $h$ is constant and the completeness of $\overline{M}$ or Remark \ref{rem_nice} imply that $I = \RR$. Moreover, $\kappa = 0$ by \eqref{stab:constcurv} and we conclude that $\PP$ is flat and $X = h_0 \partial_t$ for some $h_0 > 0$. This concludes the proof of Step 3.\\[0.3cm]
\textbf{Step 4:} if $u>0$ solves \eqref{stab:Phieq3}, then $M$ is isometric to a cylinder $\RR \times \Sigma$ for some complete $\Sigma$, with metric and second fundamental form given by
	\begin{equation}\label{locexp_goal}
		\metric = \di s \otimes \di s + \metric_{\Sigma}, \qquad \II = \mu_1(s) \, \di s \otimes \di s \otimes \nu,
	\end{equation}
where $\nu$ is a unit normal vector to $M \to \Mbar$. Moreover, $\mu_1 \neq 0$ on $\RR$, and the soliton constant satisfies $c \neq 0$.\\[0.2cm] 	
\noindent \emph{Proof of Step 4.} By Step 3, we know that $u$ cannot be constant on any open set of $M$, thus the set $\{\nabla u \neq 0 \}$ is nonempty and dense in $M$. Fix a point $p \in M$ such that $\nabla u (p) \neq 0$. The tensor field $\II$ is Codazzi, as observed in Step 2, and its traceless part $\Phi$ attains the equality in Kato's inequality by \eqref{condisplit}. By applying Lemma \ref{lem_kato} with $A = \II$, we obtain the existence of a neighbourhood $U$ of $p$ that splits as a Riemannian product $(-\varepsilon,\varepsilon) \times \Sigma^{m-1}$ and such that the metric $\metric$ of $M$ and the tensor field $\II$ can be written as
	\begin{equation} \label{locexp}
		\metric = \di s \otimes \di s + \metric_{\Sigma}, \qquad \II = \Big( \mu_1(s) \, \di s \otimes \di s + \mu_2(\pi_{\Sigma}) \metric_{\Sigma} \Big) \otimes \nu,
	\end{equation}
	for some smooth functions $\mu_1 : (-\varepsilon, \varepsilon) \to \RR$, $\mu_2 : \Sigma \to \RR$ satisfying
	\begin{equation} \label{noumb}
		\mu_1(s) \neq \mu_2(x) \quad \text{for each } \, s \in (-\varepsilon, \varepsilon), \, x \in \Sigma.
	\end{equation}
	Up to reparametrizing $(-\varepsilon,\varepsilon)$, we can write
	\begin{equation} \label{pexp}
		p = (0,q) \qquad \text{for some } \, q \in \Sigma.
	\end{equation}

	Now we prove that $\mu_1(s)\mu_2(\pi_{\Sigma}) \equiv 0$ on $U$. Let $\{\theta^i\}$ be a local orthonormal coframe on $U$ as the one described in the last part of the proof of Lemma \ref{lem_kato}. In particular, we assume that $\theta^1 = ds$ and then we have $\theta^1_j \equiv 0$ on $U$ for $1 \leq j \leq m$. Writing
	$$
	\II = a_{ij} \theta^i \otimes \theta^j \otimes \nu,
	$$
	we have
	$$
	a_{11} = \mu_1, \qquad a_{ii} = \mu_2 \quad \text{for } \, 2 \leq i \leq m \qquad \text{and} \qquad a_{ij} = 0 \quad \text{for each } \, i \neq j.
	$$
	On the other hand, since $\Mbar$ is flat, Gauss' equations give
	\begin{equation} \label{gauss}
	R_{ijkt} = a_{ik} a_{jt} - a_{it} a_{jk} \qquad \text{for } \, 1 \leq i, j, k, t \leq m
	\end{equation}
	where $R_{ijkt}$ are the components of the Riemann curvature tensor of $M$ along $\{\theta^i\}$. Recalling that $\theta^1_j \equiv 0$ for $1 \leq j \leq m$, by Cartan structural equations
	$$
	\frac{1}{2} R_{ijkt} \theta^k \wedge \theta^t = \di \theta^i_j + \theta^i_k \wedge \theta^k_j \qquad \text{for } \, 1 \leq i, j \leq m
	$$
	we immediately see that
	$$
	R_{1jkt} = 0 \qquad \text{for } \, 1 \leq j, k, t \leq m.
	$$
	Putting together these facts, from \eqref{gauss} we obtain
	$$
	\mu_1(s) \mu_2(\pi_{\Sigma}) = a_{11} a_{22} = a_{11} a_{22} - a_{12} a_{12} = R_{1212} = 0.
	$$
	
	Note that $\mu_1$ and $\mu_2$ can never be both zero at the same point by \eqref{noumb}. Since they depend on disjoint sets of variables, this implies that exactly one of them identically vanishes on its domain while the other one never attains the zero value. In the $2$-dimensional case where $m=2$ we can assume without loss of generality that $\Sigma$ is an interval and then $\mu_2 \equiv 0$, up to renaming indices. We claim that $\mu_2$ identically vanishes on $\Sigma$ also in case $m \geq 3$. Suppose, by contradiction, that $\mu_2 \neq 0$. Then $\mu_1 \equiv 0$ on $(-\varepsilon,\varepsilon)$ while $\mu_2$ has constant (by Lemma \ref{lem_kato}) nonzero value. By \eqref{locexp}, $\psi$ has constant mean curvature $H = \frac{m-1}{m}\mu_2$ in $U$. Putting this constant value of $H$ into equation \eqref{stab:Heq} we obtain $0 = |\II|^2 H^2$, contradiction. So, we have proved that $\mu_2 \equiv 0$ on $\Sigma$. Moreover, the mean curvature of $\psi$ is given by $H = \frac{1}{m}\mu_1(s) \neq 0$ on $U$, so $\psi$ is not a minimal hypersurface and therefore $c \neq 0$. Also, the rank of $\II$ is exactly $1$ in a neighbourhood of $p$. Since $u>0$, we observe that $M$ does not possess totally geodesic points, hence the rank of $\II$ it at least $1$ everywhere. Since the rank of $\II$ is lower-semicontinuous and it is $1$ on the dense subset $\{\nabla u \neq 0\}$, we deduce that $\II$ has rank one everywhere. In particular, the distribution corresponding to the nullity of $\II$ is smooth, totally geodesic and integrable (cf. \cite[Proposition 1.18]{DaTo}). Hence, the entire $M$ splits as $\RR \times \Sigma$ for some complete $\Sigma$, and the metric and second fundamental form write as claimed on the entire $M$. \\[0.3cm]
\textbf{Step 5:} if $u>0$ solves \eqref{stab:Phieq3}, then $\psi$ lifts to a (possibly tilted) grim reaper immersion $\hat \psi : \RR^m \to \RR^{m+1}$. \\[0.2cm]
\emph{Proof of Step 5.} By Step 3 and by completeness of $\Mbar$, $I \times_h \PP = \RR \times F$, with $(F, \metric_F) = (\PP, h_0^2\metric_{\PP})$ a complete flat manifold, and $\psi$ is a translating soliton with respect to the parallel vector field $\partial_t$ with soliton constant $c h_0$, which is non-zero by Step 4. Without loss of generality, we can therefore assume $c > 0$.
	
	Let $\pi : \RR^m \to F$ be the universal Riemannian covering of $F$. Then
	$$
	\pi_{\Mbar} = \mathrm{id}_{\RR} \times \pi : \RR \times \RR^m \to \RR \times F = \Mbar
	$$
	is the universal Riemannian covering of $\Mbar$. The deck transformation group of the covering $\pi$ is a discrete subgroup $\Gamma_F$ of the isometries of $\RR^m$, and $F = \RR^m/\Gamma_F$, while the deck transformation group of the covering $\pi_{\Mbar}$ consists of the maps of the form $\mathrm{id}_{\RR} \times T : \RR^{m+1} \to \RR^{m+1}$, with $T \in \Gamma_F$. Consider the immersion $\tilde \psi = \psi \circ \pi_M : \RR^m \to \Mbar$. Then, $\tilde \psi$ uniquely lifts to an immersion
	$$
	\hat\psi : \RR^m \to \RR^{m+1} \quad \text{such that} \quad \pi_{\Mbar} \circ \hat \psi = \tilde\psi.
	$$
It is easy to see that $\hat\psi$ is again a translating mean curvature flow soliton with respect to the lift $\hat\partial_t \in \mathfrak X(\RR^{m+1})$ of $\partial_t$ with soliton constant $ch_0$. Furthermore, by Step 4 the universal covering of $M$ splits as $\RR^m = \RR \times \RR^{m-1}$, where $\RR^{m-1}$ covers $\Sigma$ and the metric and second fundamental form on $\RR^m$ have the expression
	\[
		\metric = \di s \otimes \di s + \metric_{\RR^{m-1}}, \qquad \hat \II = \mu_1(s) \, \di s \otimes \di s \otimes \hat\nu,
	\]
	with $\hat\nu$ the local normal vector field along $\hat\psi$ given by the lift of $\nu$. In particular, $\Sigma$ is totally geodesic in $\RR^m$.

	From the expression of metric and second fundamental form of $M$, and from Gauss equations, we deduce that $M$ is flat. Since the nullity of $\hat\II$ has dimension $m-1$, a classical theorem of Hartman \cite{hartman} guarantees that $\hat\psi$ is a flat cylinder over a plane curve. More precisely, for each $q \in \RR^{m-1}$ we define 
	$$
	\gamma_q : \RR \times \RR^m, \qquad  s \mapsto (s,q).
	$$
Then, $\hat\psi(\gamma_q)$ is contained in the $2$-plane $\Pi_q = (\hat\psi_{\ast}T_q \RR^{m-1})^{\bot} \subseteq T_{\hat \psi(0,q)}\RR^{m+1}$, and the planes $\Pi_q$ are all parallel. We denote with $\Pi$ the plane associated to $q=0$. As observed in the proof of Theorem \ref{B:thm:euclidtrans}, $\gamma_0$ is itself a translating soliton with soliton constant $ch_0$ in $\Pi$, with respect to the orthogonal projection of the vector field $\hat\partial_t$ onto $\Pi$. We let $V$ denote such orthogonal projection and define 
	\begin{equation}\label{def_alpha_GR}
	\alpha\in[0,\pi/2) \qquad \text{such that } \, ||V|| = \cos\alpha. 	
	\end{equation}
In fact, $V$ is a nonzero vector, since otherwise $\hat \psi(\gamma_0)$ would be a straight line and $\psi$ would be totally geodesic. Then, $\hat \psi(\gamma_0)$ is a translating soliton curve with soliton constant
	\begin{equation}\label{def_k_gr}
		k = ch_0\cos\alpha > 0 
	\end{equation}
with respect to a parallel unit vector field in the Euclidean plane and therefore, under a suitable choice of cartesian coordinates $(x^1,x^2)$ on $\Pi$ such that $V = \cos\alpha \, \partial_1$, it can be reparametrized as the grim reaper curve
	\begin{equation}\label{eq_GR}
	\begin{array}{rcrcl}
	\sigma & : & \left( -\frac{\pi}{2 k}, \frac{\pi}{2k} \right) & \rightarrow & \RR^2 \\ [0.2cm]
	& & \tau & \mapsto & \sigma(\tau) = (\sigma_1(\tau), \sigma_2(\tau)) = \left( -\frac{1}{k}\log(\cos(k\tau)), \tau \right).
	\end{array}
	\end{equation}
	Indeed, a translating soliton curve with respect to $\partial_1$ in $\RR^2$ with soliton constant $k$ can always be locally written as a graph $x_1 = f(x_2)$ with $f$ satisfying
	$$
	k = \sqrt{1+(f')^2}\left(\frac{f'}{\sqrt{1+(f')^2}}\right)' = \frac{f''}{1+(f')^2} = (\arctan(f'))'.
	$$
	Hereafter, the point $\sigma(0)$ of the grim reaper curve parametrized as above will be referred to as the vertex of the grim reaper. So, $\hat\psi$ is a (possibly tilted) grim reaper cylinder, and up to an isometry of the second factor of the ambient space $\RR^{m+1}= \RR \times \RR^m$ and a translation in the first factor (which does not affect the validity of \eqref{stab:PhiL2}) we can assume 
	\begin{equation}\label{eq_hatpsi}
	\begin{array}{l}
		\hat \psi(s, x_2, x_3, \ldots, x_m) \\[0.2cm]
		\qquad \disp = \left( x_2\sin \alpha + \sigma_1(\tau(s))\cos \alpha, \sigma_2(\tau(s)), x_2\cos \alpha - \sigma_1(\tau(s))\sin \alpha, x_3, \ldots, x_m \right),
	\end{array}
	\end{equation} 
	where $\tau(s)$ is the change of parameter from arclength $s$ to $\tau$ in \eqref{eq_GR}. To introduce Step 6, let $\Omega_0 = \{0\} \times \RR^{m-1} \subset \RR^m$ be the ``valley" of the grim reaper, that is mapped by $\hat\psi$ to the vertices of the grim reaper curves $\hat \psi(\gamma_q)$, $q \in \RR^{m-1}$. Also, let $\Omega = \pi_M(\Omega_0)$. \\[0.3cm] 
\textbf{Step 6:} the grim reaper $\hat \psi$ is not tilted (that is, the angle $\alpha$ defined in \eqref{def_alpha_GR} vanishes).\\[0.2cm]
\noindent \emph{Proof of Step 6.} Suppose, by contradiction, that $\alpha \neq 0$. Let $\hat W$ be the orthogonal projection of $\hat\partial_t$ onto the subspace $\hat\psi_\ast T\Omega_0$ of $T\RR^{m+1}$, so $\hat W \neq 0$, and let $W \in T\Omega_0$ be the (never vanishing) induced vector field on $\Omega_0$. Explicitely, we have
	\begin{align*}
		\hat W & \equiv (\sin^2\alpha,0,\sin\alpha\cos\alpha,0,\dots,0) \in \RR^{m+1} \, , \\
		W & \equiv (0,\sin\alpha,0,\dots,0) \in \RR^m \, .
\end{align*}
As $W $ is parallel on $\Omega_0$, it allows to split $\Omega_0$ as the product $\Sigma_0^{m-2} \times \RR$, where the tangent vector to the $\RR$ direction is $W$. Let $\ell_W \subset \Omega_0$ be a line in the universal covering $\RR \times \Sigma_0 \times \RR$ of $M$ of the form $\ell_W(t) = (0,z,t)$ for $t \in \RR$ and fixed $z \in \Sigma_0$. Next, observe that $\hat W$ is not tangent to the ``horizontal" factor $\RR^m$ of $\RR^{m+1} = \RR \times \RR^m$, so for any deck transformation $\hat T = \mathrm{id}_\RR \times T \in \mathrm{deck}(\pi_{\overline{M}})$, condition
$\hat T(\hat \psi(\RR^m)) = \hat \psi(\RR^m)$ forces the map $T$ to act as the identity in the direction given by the projection of $\hat W$ on the horizontal $\RR^m$. In particular, this implies the following properties:
\begin{equation}\label{eq_prope_W}
\begin{array}{ll}
(i) & \text{there exists no nontrivial deck transformation of $\pi_{\Mbar}$ fixing $\hat\psi(\ell_W)$;} \\[0.2cm]
(ii) & (\pi_{\overline{M}})_* \hat W = \hat W; \\[0.2cm] 
(iii) & \text{$\pi_{\Mbar}$ is injective on $\hat \psi(\ell_W)$.}
\end{array} 
\end{equation}
From $(iii)$, $\psi\circ\pi_M = \pi_{\Mbar} \circ \hat \psi$ is injective on $\ell_W$, since $\hat\psi$ is injective. As a consequence, $\pi_M$ is injective on $\ell_W$ and $\psi$ is injective on $\pi_M(\ell_W)$. Note also that $\psi\circ\pi_M(\ell_W) = \pi_{\Mbar} \circ \hat \psi (\ell_W)$ is a proper curve in $\overline{M}$, and thus the curve $\pi_M(\ell_W)$ is proper in $M$ (and contained into $\Omega$). Indeed, for each compact set $K \subset M$, 
	\[
	(\pi_M \circ \ell_W)^{-1}(K) \subset (\psi \circ \pi_M \circ \ell_W)^{-1}(\psi(K)) 
	\]
is therefore closed in a compact set, hence it is compact. Summarizing, $\pi_M(\ell_W)$ is a proper and injective immersion, hence a proper embedding. Next, observe that if two lines $\ell_W,\ell'_W$ are different (hence, they do not intersect), then either $\pi_M(\ell_W) \cap \pi_M(\ell'_W) = \emptyset$ or $\pi_M(\ell_W) = \pi_M(\ell'_W)$. Indeed, if $\pi_M(\ell_W)$ and $\pi_M(\ell'_W)$ intersect but do not coincide, since they are geodesics in $M$ they have to intersect transversely. However, the two curves $\hat \psi(\ell_W)$ and $\hat \psi(\ell'_W)$ are parallel straight lines on $\RR^{m+1} = \RR \times \RR^m$ in the direction of $\hat W$, hence by $(ii)$ in \eqref{eq_prope_W} their projections cannot be transverse. 

Set $\Sigma = \pi_M(\Sigma_0) \subset \Omega$, and take a small, contractible and relatively compact open subset $\Sigma'$ of $\Sigma$. By a compactness argument that uses the properness of $\pi_M(\ell_W)$, up to reducing $\Sigma'$ we can guarantee that each $\pi_M(\ell_W)$ meets $\Sigma'$ at most once. Summarizing the above properties, since $\pi_M(\ell_W)$ is a geodesic, the exponential map $\exp^\perp : T\Sigma'^\perp \to \Omega$ realizes a diffeomorphism between $T\Sigma'^\perp$ and the union $E \subset \Omega_0$ of all curves $\pi_M(\ell_W)$ passing through $\Sigma'$. Note that $T\Sigma'^\perp$ is diffeomorphic to $\Sigma' \times \RR$ since $\Sigma'$ is contractible. From $\Omega_0 = \Sigma_0 \times \RR$ and the constructions of $\ell_W$ and $\Sigma'$, we deduce that the pulled-back metric via $\exp^\perp$ on $T\Sigma'^\perp$ is the product metric, whence $E$ is isometric to $\Sigma' \times \RR$. In conclusion, $M$ contains a subset $M_0$ of positive $m$-dimensional measure that splits as $\RR \times E = \RR \times \RR \times \Sigma'$. Fix $U\subseteq\Sigma_0$ an open subset such that $\pi_{M|U} : U \to \Sigma'$ is a diffeomorphism. Then
	$$
	\pi_{M|\RR^2\times U} : \RR \times \RR \times U \to M_0
	$$
	is an isometry, and from 
	\begin{equation}\label{eq_perinte}
	\begin{array}{rcl}
	\etabar(\pi_I) & = & \disp h_0(\pi_I-t_0) \quad \text{for some } \, t_0 \in \RR, \\[0.3cm]
	|\Phi|^2 & = & \disp |\II|^2 - 2H^2 = (\mu_1)^2 - 2\left(\frac{\mu_1}{2}\right)^2 = \frac{(\mu_1)^2}{2} = \frac{1}{2}\left(\frac{k}{||\dot\sigma||}\right)^2, \\[0.3cm]
	||\dot\sigma(\tau)||^2 & = & \disp 1 + \tan^2(k\tau) = \frac{1}{\cos^2(k\tau)}.
	\end{array}
	\end{equation}
we can estimate
	\begin{align*}
			\int_M |\Phi|^2 e^{c\eta} \geq \int_{M_0} |\Phi|^2 e^{c\eta} & = |\Sigma'| \int_{\RR^2} \frac{\mu_1(s)^2}{2} e^{ch_0(x_2\sin\alpha + \sigma_1(\tau(s))\cos\alpha - t_0)} \, \di s \, \di x_2 \\
			& = |\Sigma'| \frac{k\pi e^{-ch_0t_0}}{2}  \int_{\RR} e^{(ch_0\sin\alpha) x_2} \, \di x_2 = +\infty
		\end{align*}
		where $|\Sigma'|$ is the $(m-2)$-dimensional volume of $\Sigma'$ and we have used \eqref{def_k_gr} and 
		\begin{equation}\label{eq_useful_mu}
			\begin{array}{lcl}
				\disp \int_{\RR} \mu_1(s)^2 e^{k\sigma_1(\tau(s))} & = & \disp k^2 \int_{\RR} \cos^2(k\tau(s)) e^{-\log(\cos(k\tau(s)))} \, \di s = k^2 \int_{\RR} \cos(k\tau(s)) \, \di s \\[0.4cm]
				& = & \disp k^2 \int_{\RR} \frac{\di s}{\|\dot\sigma(\tau(s))\|} = k^2 \int_{\RR} \tau'(s) \, \di s = k^2 \int_{-\pi/(2k)}^{\pi/(2k)} \di \tau = k\pi \, .
			\end{array}
	\end{equation}
	So, $|\Phi| \not\in L^2(M,e^{c\eta})$, contradicting \eqref{stab:PhiL2}.\\[0.3cm] 
\textbf{Step 7:} if $u>0$ solves \eqref{stab:Phieq3}, then there exists a Riemannian submersion $\pi_\Omega : M \to \Omega$ with $1$-dimensional, noncompact, totally geodesic fibers that are sent, via $\hat \psi$, to the grim reaper curves \eqref{eq_GR}. \\[0.2cm]
\noindent \emph{Proof of Step 7.} Having shown that $\alpha = 0$, $\hat \psi$ writes as 
	\begin{equation}\label{eq_hatpsi_2}
	\hat \psi(s, x_2, x_3, \ldots, x_m) = \left( \sigma_1(\tau(s)), \sigma_2(\tau(s)), x_2, x_3, \ldots, x_m \right).
	\end{equation} 
We first describe the structure of translations $\hat T = \mathrm{id}_{\RR} \times T \in \mathrm{deck}(\pi_{\overline{M}})$. Choose Cartesian coordinates $(t,y_1,y_2,\ldots, y_m)= (t,y_1,y')$ on $\RR^{m+1}$, so by \eqref{eq_hatpsi_2} the image of $\hat \psi$ can be written as $t = \sigma_1(\tau(s))$, $y_1=\sigma_2(\tau(s))$. To preserve $\hat \psi(\RR^m)$, the component $T$ of $\hat T$ shall satisfy $T(y_1,y') = (\pm  y_1, T'(y_1,y'))$, since the image of lines obtained by fixing $y_1$ have bounded $y_1$ coordinate and since $y_1(\hat \psi(\RR^m))$ is invariant by $\hat T$. As $T$ is an isometry, $T'$ only depends on $y'$. In particular, either $T(\hat \psi(\gamma_q)) = \hat \psi(\gamma_{q'})$ or $T(\hat \psi(\gamma_q)) \cap \hat \psi(\gamma_{q'}) = \emptyset$. Therefore, either $\pi_{\overline{M}}(\hat \psi(\gamma_{q}))$ and $\pi_{\overline{M}}(\hat \psi(\gamma_{q'}))$ coincide or they have empty intersection, in particular, they cannot be transverse.	

We next study the geodesics $\pi_M(\gamma_q)$. First, again since deck transformations of $\pi_{\Mbar}$ act as the identity on the first component, $\pi_{\Mbar}\circ \psi (\gamma_{q})$ is a proper curve in $\Mbar$, thus, as in Step 6, $\pi_M(\gamma_q)$ is a proper geodesic in $M$. We claim that $\pi_M(\gamma_q)$ is injectively immersed, hence embedded because of its properness. If $\pi_M(s_1,q) = \pi_M(s_2,q)$ for some $s_1 \neq s_2$, then $\hat \psi(s_1,q)$ and $\hat \psi(s_2,q)$ project onto the same point in $\Mbar$. Comparing the first component of the two points in $\RR^{m+1}$, we deduce $s_1 = \pm s_2$, hence $s_1 = -s_2$. Let $\hat T \in \mathrm{deck}(\pi_{\overline{M}})$ be the deck transformation that satisfies $\hat T(\hat \psi(s_1,q)) = \hat \psi(-s_1,q)$, and consider the geodesic $\sigma \subset \RR^{m+1}$ joining $\hat \psi(s_1,q)$ and $\hat \psi(-s_1,q)$. Then, $\hat T(\sigma) = \sigma$, and since $\hat T(\hat \psi(\RR^m)) = \hat \psi(\RR^m)$ the middle point of the segment joining $\hat \psi(s_1,q)$ and $\hat \psi(-s_1,q)$ shall necessarily be a fixed point of $\hat T$. Hence $\hat T = \mathrm{id}$, contradiction.

	Having proved that $\pi_M(\gamma_q)$ is injectively immersed, we claim that for $q \neq q'$ either $\pi_M(\gamma_q) \cap \pi_M(\gamma_{q'}) = \emptyset$ or they coincide. We proceed by contradiction: since both the curves are geodesics, we assume that $\pi_M(\gamma_q)$ and $\pi_M(\gamma_{q'})$ are transverse somewhere. Then, also their images $\psi \circ \pi_M(\gamma_q)= \pi_{\overline{M}} \circ \hat \psi(\gamma_q)$ and $\psi \circ \pi_M(\gamma_{q'})= \pi_{\overline{M}} \circ \hat \psi(\gamma_{q'})$ are transverse somewhere, which contradicts the observations at the beginning of this Step.  

With the above preparation, let $x \in M$ and $(s,q),(s',q') \in \pi_M^{-1}(x)$. From the fact that $\pi_M(\gamma_q)$ and $\pi_M(\gamma_{q'})$ either coincide or they do not intersect, we also deduce that $\pi_M(0,q) = \pi_M(0,q')$. Hence, the map 
	\[
	\pi_\Omega : M \to \Omega, \qquad \pi_\Omega(x) = \pi_M\big(  (0,q) \big) \quad \text{for any chosen } \, (s,q) \in \pi_M^{-1}(x)
	\]
is well defined. Fix a contractible, small open subset $\Omega' \subset \Omega$, and let $\Omega_0'\subset \Omega_0$ be one of its diffeomorphic lifts by $\pi_M$. Proceeding as in the end of Step 6, we can prove that the union $E \subset M$ of lines $\pi_M(\gamma_q)$ passing through points $q \in \Omega'$ is isometric to $\RR \times \Omega_0'$, the isometry being $(s,q) \to \pi_M(\gamma_q(s))$. It follows that $\pi_\Omega$ is a fibration and a Riemannian submersion. \\[0.3cm]
\textbf{Step 8: } $\Omega$ is compact.\\[0.2cm]
\noindent \emph{Proof of Step 8.} A straightforward computation that uses \eqref{eq_perinte} and \eqref{eq_useful_mu} then shows
	\begin{align*}
	\int_{M} |\Phi|^2 e^{c\eta} & = |\Omega| e^{-ch_0t_0} \int_{\RR} \frac{\mu_1(s)^2}{2} e^{k\sigma_1(\tau(s))} \di s \\
	& = |\Omega| \frac{\pi k e^{-ch_0t_0}}{2},
	\end{align*}
	where $|\Omega|$ is the $(m-1)$-dimensional volume of $\Omega$. So in this case $|\Phi| \in L^2(M,e^{c\eta})$ holds true if and only if the manifold $\Omega$ has finite volume. Being $\Omega$ flat, $\Omega$ must be compact.\\[0.3cm] 
\textbf{Step 9: } Each of the solitons $\psi$ in Item (ii) is stable.\\[0.2cm]
\noindent \emph{Proof of Step 9.} It is known that the embedding $\hat \psi : \RR^m \to \RR^{m+1}$ is stable: to show this, denoting with $\hat \nu$ a global choice of the normal vector, it is enough to observe that the function $\hat v = \langle \hat \nu,\hat \partial_t \rangle$ has a sign, say it is positive up to suitably choosing $\hat \nu$, and satisfies $0 = \Delta_{-c\eta} \hat v + |\II|^2 \hat v = L \hat v$ on $M$ by Proposition \ref{B:prp:uequation}. Let $T : \RR^m \to \RR^m$ be a deck transformation of $\pi_M$. For fixed $x \in M$, we compare $\hat v(\tilde x)$ to $\hat v(T(\tilde x))$, where $\tilde x \in \pi_M^{-1}(x)$. Since every deck transformation of $\pi_{\overline{M}}$ acts as the identity in the first factor of $\RR \times \RR^m$, the product of $\hat \nu$ with $\hat \partial_t$ is constant on the fiber $\pi_{\overline{M}}^{-1}(\psi(x))$. Therefore, since $\hat \psi(\tilde x), \hat \psi(T(\tilde x)) \in \pi_{\overline{M}}^{-1}(\psi(x))$, we deduce that $\hat v(\tilde x) = \hat v(T(\tilde x))$, hence $\hat v$ induces a smooth, positive function $v : M \to \RR$ which solves $L v = 0$, proving the stability of $\psi$. 
\end{proof}

We can apply the above result, for instance, to the case $\Mbar^{m+1}=\HH^{m+1}$ the hyperbolic space, to deduce the following half-space theorem:
\begin{corollary} \label{stab:crl:hspaceumbilicity}
	There exist no complete stable mean curvature flow solitons $\psi : M \to \HH^{m+1} = \RR\times_{e^t}\RR^m$ with respect to $e^t \partial_t$ with soliton constant $c<0$ and satisfying
	\begin{equation} \label{stab:hspacebound1}
	|\Phi| \in L^2(M) \quad \text{and} \quad \psi(M) \subseteq \left[ \log\left(-\frac{m}{c}\right), + \infty\right),
	\end{equation}
	with $\Phi$ the umbilicity tensor of $\psi$.
\end{corollary}

\begin{proof}
	By contradiction. Suppose that there exists $\psi : M \to \Mbar$ satisfying the given conditions. Noting that \eqref{stab:PhiL2} is implied by $|\Phi| \in L^2(M)$ and the fact that $\psi$ lies in the given half-space, we apply Theorem \ref{stab:thm:umbilicity} to each connected component of $M$ to conclude that $\psi$ is a totally geodesic complete hyperplane containing $\partial_t$. This contradicts the hypothesis that $\psi(M)$ is contained in a vertical half-space.
\end{proof}

\begin{remark} \label{remhoro}
	As a byproduct of Corollary \ref{stab:crl:hspaceumbilicity}, note that the horosphere $\left\{ t = \log\left(-\frac{m}{c}\right) \right\}$ is a posteriori an unstable soliton.
\end{remark}

\begin{proof}[Proof of Corollary \ref{intro:crl:sphere}] 
	We represent $\sphere^{m+1} \backslash \{p,-p\}$ as the warped product $(0, \pi) \times_{\sin t} \sphere^{m}$, with $(t,x) \to p$ as $t \to 0$ for each $x \in \sphere^m$. Note that the coordinate $t$ on the $(0,\pi)$-factor is everywhere equal to the distance from $p$ on $\sphere^{m+1} \backslash \{p,-p\}$. Suppose by contradiction that $\psi$ is stable. Then, we apply Theorem \ref{stab:thm:umbilicity}: note that the issue of $M$ possibly passing through $p,-p$ is easily handled by observing that the relevant differential inequalities for $|\Phi|^2$ and $H^2$ also hold in $\psi^{-1}\{p,-p\}$, and the completeness of $\overline{M}$ is only used in case (ii) to characterize $M$ as a quotient of the grim-reaper. Therefore, we can infer that $\psi$ is totally geodesic, hence an equator. However, in this case the stability operator becomes
	$$
	L = \Delta_{-c\eta} + m - c \cos(\pi_I), 
	$$
	and it is always unstable being $M$ compact, $m - c \cos(\pi_I) \ge 0$ and positive somewhere. Contradiction.
\end{proof}

The next result estimates the number of large ends of a soliton $M$, in a suitable sense recalled below, under the assumption that $M$ is stable or has finite stability index. We recall that the study of the ends of $M$ in terms of harmonic (or, more generally, $\Delta_f$-harmonic) functions with special properties on $M$ is a beautiful theory applied with remarkable success due to P. Li, L.F. Tam and J.P. Wang, see \cite{Li_book} for a detailed presentation and \cite{PRS} for recent improvements. We recall some terminology: taking an exhaustion $\Omega_j$ of $M$ by relatively compact open sets, $M$ is said to have finitely many ends if the number of ends of $M \backslash \overline{\Omega}_j$ (non-decreasing as a function of $j$) has a finite limit as $j \ra \infty$. The limit is independent of the chosen exhaustion. Given an end $E$ with $\partial E$ smooth, a double of $E$ is a manifold obtained by gluing two copies of $E$ along $\partial E$ the common boundary, with metric being that of $E$ apart from a neighbourhood of $\partial E$; then, $E$ is said to be $\Delta_f$-hyperbolic if $\Delta_f$ admits a positive Green kernel on some (equivalently, any) double of $E$. The notion of $\Delta_f$-hyperbolicity is, roughly speaking, a way to measure the size of $E$.

\begin{proposition} \label{stab:prp:hypends}
	Let $\psi:M^m\to\Mbar^{m+1}=I\times_h\PP$ be a connected, complete mean curvature flow soliton with respect to $X=h(t)\partial_t$ with soliton constant $c$. Assume that $\Mbar$ has constant sectional curvature $\bar\kappa$ and that 
	\begin{equation} \label{stab:largekappabar}
	c h'(\pi_I \circ \psi) \le \frac{2m-1}{2} \bar \kappa \qquad \text{on } \, M \backslash\Omega,
	\end{equation}
	for some relatively compact open set $\Omega$. Then, 	
	setting $\eta=\etabar\circ\psi$ with $\etabar$ as in \eqref{B:defetabar},
	\begin{itemize}
		\item[(i)] if $M$ is stable and \eqref{stab:largekappabar} holds with $\Omega = \emptyset$, then $M$ has only one $\Delta_{-c\eta}$-hyperbolic end;  		
		\item[(ii)] if $M$ has finite index, then $M$ has at most finitely many $\Delta_{-c\eta}$-hyperbolic ends.
	\end{itemize}	
\end{proposition}

\begin{proof}
	Since $\Mbar$ has constant sectional curvature $\bar\kappa$, $\PP$ has constant sectional curvature, say, $\kappa$, with $\bar \kappa$ and $\kappa$ related by Gauss equation \eqref{stab:constcurv}. Consider the Bakry-Emery Ricci tensor
	\[
	\Ric_{-c\eta}=\Ric-c\Hess(\eta).
	\]
	Then by the inequality before (5.13) of \cite{AdeLR} we have
	\begin{equation}\label{eq_lowerricci}
	\Ric_{-c\eta}\geq-(|\II|^2+ch'-(m-1)\bar\kappa)\metric_M \doteq \rho \metric_M.
	\end{equation}
	In this setting the stability operator (\ref{stab:stoperator}) becomes
	\begin{equation}
	L = \Delta_{-c\eta} + (|\II|^2-ch'+m\bar\kappa).
	\end{equation}
	Under assumptions $(i)$ or $(ii)$, without loss of generality we can assume that $0 \leq \lambda_1^L(M \backslash \overline{\Omega})$, and thus there exists a positive smooth solution of $Lu=0$ on $M \backslash \overline{\Omega}$.
	Hence, considering the operator
	\[
	\bar L = \Delta_{-c\eta} - \rho 
	\]
	from \eqref{eq_lowerricci} and \eqref{stab:largekappabar} we have
	\[
	\begin{array}{lcl}
	\bar L u & = & \disp \Delta_{-c\eta} u + (|\II|^2 + ch'-(m-1)\bar\kappa)u = L u + (2ch' -(2m-1) \bar \kappa) u \\[0.2cm]
	& = & (2ch' -(2m-1) \bar \kappa) u \le 0 \qquad \text{on } \, M \backslash \overline{\Omega}.
	\end{array}
	\]
	This means that $\lambda^{\bar L}_1(M\backslash \overline\Omega)\geq 0$, and therefore $\bar L$ is nonnegative (in case (i)) or has finite index (in case (ii)). Applying Corollary 7.12 of \cite{PRS} we deduce that $M$ has at most one (respectively, finitely many) $\Delta_{-c\eta}$-hyperbolic ends.
\end{proof}

\begin{remark}
	Corollary 7.12 of \cite{PRS} is stated in terms of the standard Laplacian $\Delta$, but its extension to 
	the weighted operator $\Delta_f$ is merely notational. Moreover, the result requires the function $\rho$ in \eqref{eq_lowerricci} be non-negative, but indeed this is not necessary (although, the case of non-negative $\rho$ is the most relevant one). In fact, Corollary 7.12 is a direct application of Theorem 5.1 in \cite{PRS}, where no sign assumption on $\rho$ is needed. 
\end{remark}

We conclude with a result that establishes a sufficient condition for every end of $M$ to be $\Delta_{-c\eta}$-hyperbolic:

\begin{proposition} \label{stab:prp:isoperimetric}
	Let $\psi:M^m\to\Mbar^{m+1}=I\times_h\PP$ be a connected, complete mean curvature flow soliton with respect to $X=h(t)\partial_t$ with soliton constant $c$. Assume that $\Mbar$ is a Cartan-Hadamard manifold (i.e. simply connected, with nonpositive sectional curvature), and that, setting $\eta = \bar \eta \circ \psi$ with $\bar \eta$ as in \eqref{B:defetabar},
	\begin{equation}\label{ipo_sob}
	\inf_{M} e^{c\eta} > 0, \qquad |\overline\nabla \bar \eta| \in L^m( M, e^{c\eta}).
	\end{equation}
	Then, every end of $M$ is $\Delta_{-c\eta}$-hyperbolic.
\end{proposition}

\begin{proof}
	The proof directly follows from results in \cite{imperimo_1}, so we will be sketchy. The first in \eqref{ipo_sob} guarantees that
	$$
	\inf_{x \in M} \liminf_{r \ra 0} \frac{\vol_{-c\eta}\big(\partial B_r(x)\big)}{r^m} > 0
	$$
	with $\vol_{-c\eta}$ the $(m-1)$-dimensional measure weighted by $e^{c\eta}$, and $\partial B_r(x)$ the sphere in the metric of $M$. Therefore, $M$ enjoys the Sobolev inequality in \cite[Thm. 6]{imperimo_1} (with the observation that the soliton equation \eqref{intro:soliton} implies the vanishing of term $\mathbf{H}_f$ in \cite{imperimo_1}):
	\begin{equation}\label{eq_sob_conpot}
	\left( \int_M |\phi|^{\frac{m}{m-1}}e^{c\eta} \right)^{\frac{m-1}{m}} \le S_1 \int_M \Big[|\nabla \phi|  + \phi |\overline \nabla \bar \eta|\Big] e^{c\eta} \qquad \text{holds } \, \forall \, \phi \in \Lip_c(M),
	\end{equation}
	for some constant $S_1$. By the second in \eqref{ipo_sob}, we can choose a compact set $K$ such that 
	$$
	\left\{ \int_{M \backslash K} |\overline \nabla \bar \eta|^m e^{c\eta} \right\}^{\frac{1}{m}} < S_1^{-1}. 
	$$	
	Restricting \eqref{eq_sob_conpot} to $\phi \in \Lip_c(M \backslash K)$ and applying Holder's inequality to the right-hand side to absorb the potential term into the left-hand side, we deduce 
	\begin{equation}\label{eq_sob_conpot2}
	\left( \int_M |\phi|^{\frac{m}{m-1}}e^{c\eta} \right)^{\frac{m-1}{m}} \le S_2 \int_M |\nabla \phi| e^{c\eta} \qquad \text{holds } \, \forall \, \phi \in \Lip_c(M\backslash K),
	\end{equation}
	for some $S_2>0$. The validity of an isoperimetric inequality on $M \backslash K$ force all ends of $M$ with respect to $K$ to be $\Delta_{-c\eta}$-hyperbolic, as observed by H.-D. Cao, Y. Shen and S. Zhu \cite{CSZ} and by P. Li and J. Wang \cite{LW} (cf. \cite[Cor. 7.17]{PRS}, and Remark 6 in \cite{imperimo_1}).
\end{proof}	

\begin{remark}
	A note of warning: in the recent \cite[Lem. 4.1]{IR}, a different kind of Sobolev inequality has been claimed for translators in Euclidean space as a consequence of the fact that translators in $\RR^{m+1}$ generate minimal hypersurfaces in $N = \RR^{m+1} \times \sphere^1$ endowed with a suitable warped product metric. However. it seems to us that the proof has a flaw, namely, to be able to apply the Hoffman-Spruck Sobolev inequality \cite{HS} the authors would need a uniform control on the injectivity radius, since $N$ is not a Cartan-Hadamard manifold, and the latter seems difficult to achieve without further assumptions.
\end{remark}

\section{MCF graphs}

In this section we consider the case where $\psi:M\to\Mbar=I\times_h\PP^m$ is the graph of a function $v:\PP\to I\subseteq\RR$, that is,
\begin{equation} \label{gr:psi}
	\psi(x)= (v(x),x)
\end{equation}
for $x\in\PP$ and, of course, $M$ is $\PP$ endowed with the graph metric $\psi^*\metric_{\Mbar}$. Indicate with $D$ the covariant derivative in $(\PP, \metric_{\PP})$. It is convenient to express the graph in terms of the flow parameter $s$ of the conformal field $X = h(t) \partial_t$:
\begin{equation}\label{def_st}
	s(t)= \int_{t_0}^t \frac{\di \sigma}{h(\sigma)}.
\end{equation}
for a fixed $t_0 \in I$, and thus to define
\begin{equation} \label{gr:defw}
	u(x)= s(v(x)), \qquad \lambda(s) = h(t(s)).
\end{equation}
Hereafter, we write $\psi_u$ to specify that the immersion $\psi$ is a graph of a function $u$ as above.

Denoting with $\nu$ the upward normal 
$$
	\nu = \frac{1}{\lambda(u)\sqrt{1+|Du|^2}} \left( \partial_s - (\Phi_u)_* Du\right),
$$
where $\Phi_u$ is the flow of $X$ at time $u(x)$, a computation shows that the (normalized) scalar mean curvature of $\psi_u$ in the direction of $\nu$ satisfies
$$
	m \lambda(u)H = \div_\PP \left( \frac{Du}{\sqrt{1+|Du|^2}}\right) - m \frac{\lambda_s(u)}{\lambda(u)} \frac{1}{\sqrt{1+|Du|^2}}.
$$
with $\lambda_s = d\lambda/ds$. If $\psi_u$ is a soliton with respect to $\partial_s = h(t) \partial_t$ with constant $c$, then
$$
	mH = c \langle \partial_s, \nu \rangle = \frac{c|\partial_s|^2}{\lambda(u)\sqrt{1+|Du|^2}} = \frac{c \lambda(u)}{\sqrt{1+|Du|^2}}.
$$
Thus, a soliton for $\partial_s$ with constant $c$ shall satisfy  
\begin{equation}\label{eq_soliton_graph}
	\div_\PP \left( \frac{Du}{\sqrt{1+|Du|^2}}\right) = \frac{c  \lambda^3(u) + m\lambda_s(u)}{\lambda(u)} \frac{1}{\sqrt{1+|Du|^2}} \doteq \frac{f(u)}{\sqrt{1+|Du|^2}}
\end{equation}
with
\begin{equation}\label{def_fs}
	f(s) = \frac{c  \lambda^3(s) + m\lambda_s(s)}{\lambda(s)} = (c h^2 + m h')(t(s)) = \zeta_c\big( t(s)\big).
\end{equation}

\begin{example}\label{ex_schwarz_constant}
	Referring to Example \ref{ex_schwarz}, we search for slices that are solitons in Schwarzschild and ADS-Schwarzschild spaces. Note that in each of \eqref{def_Vr_schwarz} and \eqref{def_Vr_ADS} the potential $V(r)$ is strictly increasing on $(r_0(\mathfrak{m}), \infty)$, therefore $h$ satisfies
	$$
	h'(t) = \frac{\di r}{\di t} = \sqrt{V(r(t))} >0, \qquad h''(t) = \frac 12 \frac{\di V}{\di r}(r(t)) > 0.
	$$
	Thus, for $c \ge 0$, the soliton function $\zeta_c(t)$ is positive and strictly increasing on $\RR^+$, hence there are no constant solutions of \eqref{eq_soliton_graph}. On the other hand, if $c<0$ slices $\{t=t_1\}$ that are solitons with respect to $h(t) \partial_t = r \sqrt{V(r)} \partial_r$ correspond to $t_1 = t(r_1)$ with $r_1> r_0(\mathfrak{m})$ solving
	\begin{equation} \label{potenziale}
	V(r) = \frac{c^2}{m^2} r^4.
	\end{equation}
	For the Schwarzschild space \eqref{def_Vr_schwarz}, such a solution do exist if and only if
	\begin{equation}\label{eq_sch}
	r_* \doteq \left( \frac{\mathfrak{m}(m-1)m^2}{2c^2} \right)^{\frac{1}{m+3}} \ge \left( \frac{\mathfrak{m}(m+3)}{2}\right)^{\frac{1}{m-1}}.
	\end{equation}
	In particular there are two solutions $r_0(\mathfrak{m}) < r_{1,-} < r_* < r_{1,+}$ if the strict inequality holds in \eqref{eq_sch}, and a unique solution $r_1 = r_*$ if equality holds.\\
	In the case of the ADS-Schwarzschild space define 
	$$
	\triangle \doteq  (m+1)^2 + 4 \bar \kappa(m-1)\frac{c^2}{m^2}(m+3).
	$$
	If $\bar\kappa \geq 0$ then $\triangle > 0$ and a solution $r_1$ of \eqref{potenziale} exists if and only if
	$$
	r_*^2 \doteq \frac{m^2}{c^2} \cdot \frac{m+1 + \sqrt{ \triangle }}{2(m+3)}
	$$
	satisfies
	$$
	\frac{\mathfrak{m}(m-1)}{r_*^{m+1}} \le \frac{2c^2}{m^2}r_*^2-1,
	$$
	and, again, there are two solutions $r_0(\mathfrak{m}) < r_{1,-} < r_* < r_{1,+}$ if the strict inequality holds, and a unique solution $r_1 = r_*$ if equality holds. Eventually, for the ADS-Schwarzschild space with $\bar\kappa < 0$, $r_1$ exists if and only if $\triangle \geq 0$ and
	$$
	r_{*,\pm}^2 \doteq \frac{m^2}{c^2} \cdot \frac{m+1 \pm \sqrt{ \triangle }}{2(m+3)}
	$$
	satisfy
	\begin{equation} \label{rastpm}
	\frac{\mathfrak{m}(m-1)}{r_{*,-}^{m+1}} \ge \frac{2c^2}{m^2}r_{*,-}^2-1, \qquad \frac{\mathfrak{m}(m-1)}{r_{*,+}^{m+1}} \le \frac{2c^2}{m^2}r_{*,+}^2-1.
	\end{equation}
	There are, indeed, two solutions $r_{1,\pm}$ satisfying \eqref{potenziale} with $r_0(\mathfrak{m}) < r_{1,-} < r_{*,+} < r_{1,+}$ if both of \eqref{rastpm} hold with the strict inequality and a unique solution $r_1 = r_{*,+}$ if the equality holds. Observe that, both in the Schwarzschild and in the ADS-Schwarzschild case, for $|c|$ positive and small enough two soliton slices do exist.
\end{example}

Equations of the type \eqref{eq_soliton_graph} have recently been investigated in \cite{BMPR}, where the authors established the following sharp weak maximum principles and  Liouville type theorems under mild growth assumptions on the volume of geodesic balls (cf. Thms. 5.4 and 8.4 therein). We report the statement to help the reading.

\begin{theorem}\label{gr:thm:BMPR}
	Let $(M, \metric)$ be complete, fix an origin $o \in M$ and let $r(x) = \dist(x,o)$. Let $u$ solve
	\[
	\div_{\PP}\left(\frac{Du}{\sqrt{1+|Du|^2}}\right) \ge \frac{f(u)}{\sqrt{1+|Du|^2}} \qquad \text{on }\, M,
	\]
	for some $f \in C(\RR)$. 
	\begin{itemize}
		\item[(i)] \cite[Thm. 5.4]{BMPR} Suppose that
		$$
		f(t) \ge C >0 \qquad \text{for } \, t >>1, 
		$$
		for some constant $C>0$, that
		$$
		\lim_{r(x) \ra +\infty} \frac{u_+(x)}{r(x)^\sigma} = 0
		$$
		for some $\sigma \in [0, 2]$, and that
		$$
		\begin{array}{ll}
		\disp \liminf_{r \ra +\infty} \frac{\log\vol(B_r)}{r^{2-\sigma}} < \infty & \quad \text{if } \, \sigma < 2, \quad \text{or} \\[0.5cm]
		\disp \liminf_{r \ra +\infty} \frac{\log\vol(B_r)}{\log r } < \infty & \quad \text{if } \, \sigma = 2. \\[0.5cm]
		\end{array}
		$$
		Then, $u$ is bounded from above and $f(\sup_M u) \le 0$.
		\item[(ii)] \cite[Thm. 8.4]{BMPR} Suppose that
		$$
		f(t) \ge Ct^\omega >0 \qquad \text{for } \, t >>1, 
		$$
		for some constants $C>0$ and $\omega>1$, and that
		$$
		\liminf_{r \ra +\infty} \frac{\log\vol(B_r)}{r^2} < \infty.
		$$
		Then, $u$ is bounded from above and $f(\sup_M u) \le 0$.\end{itemize} 
\end{theorem}

We apply the above result for the Schwarzschild spaces in Example \ref{ex_schwarz} to deduce Theorem \ref{intro:sch} in the Introduction.

\begin{theorem}
	There exists no entire graph in the Schwarzschild and ADS-Schwarz\-schild space (with spherical, flat or hyperbolic topology) over a complete $\PP$ that is a soliton with respect to the field $r \sqrt{V(r)} \partial_r$ with constant $c \ge 0$. 
\end{theorem}

\begin{proof}
	We refer to Examples \ref{ex_schwarz} and \ref{ex_schwarz_constant}. Let $\psi(x) = (v(x),x)$ be a soliton graph, and let $u(x) = s(v(x))$ be, as above, a description of the graph with respect to the flow parameter $s(t)$. We observed in Example \ref{ex_schwarz_constant} that the soliton function $\zeta_c(t)$ is positive and increasing on $\RR^+$, hence so is $f$ in \eqref{eq_soliton_graph}. If $\PP$ is compact, a contradiction directly follows by integrating \eqref{eq_soliton_graph} on $\PP$ and applying the divergence theorem. By Myers theorem, this case includes both the Schwarzschild and the ADS-Schwarzschild space with spherical topology. In the remaining cases, observe that $V(r) \sim r$ for large $r$, thus in view of \eqref{def_tr_sch}
	$$
	h(t) = r(t) \sim e^t \qquad \text{as } \, t \ra \infty.
	$$
	As a consequence, the flow parameter $s(t)$ in \eqref{def_st} is bounded from above and so is $u$. Furthermore, in the ADS-Schwarzschild space with flat or hyperbolic topology, $\PP$ is Einstein with non-positive Einstein constant, and by the Bishop-Gromov comparison theorem
	$$
	\limsup_{r \ra \infty} \frac{\log \vol B_r}{r} < \infty.
	$$
	Applying Theorem \ref{gr:thm:BMPR} we deduce $f(\sup_M u) = \zeta_c(s(\sup_M u)) \le 0$, contradiction.
\end{proof}

In a similar way, we now give the

\begin{proof}[Proof of Corollary \ref{intro:crl:hspaceslice}]
	If $\psi : M^m \to \RR \times_{e^t} \RR^m$ is a soliton graph with respect to $e^t\partial_t$ with constant $c \in \RR$, then the soliton function of $\psi$ is $\zeta_c(t) = e^t(m+ce^t)$. Up to translation, the flow parameter $s$ in \eqref{def_st} is $s = -e^{-t}$, so $u(x) = s(v(x)) <0$  on $\PP$. Setting $f(s) = \zeta_c(t(s))$ we note that $f(s) \ra m+c$ as $s \ra 0$, and applying Theorem \ref{gr:thm:BMPR} we deduce $f(\sup_M u) \le 0$, a contradiction since $\zeta_c>0$. If $c<0$, then $\zeta_c(t) = 0$ if and only if $t = t_0 \doteq \log\left(-\frac{m}{c}\right)$. We assume that $\psi(M)$ is contained into one of the two half-spaces determined by $\{t=t_0\}$, otherwise the conclusion is immediate. Applying Theorem \ref{gr:thm:BMPR} both to $u$ (if $u < s(t_0)$) or to $-u$ (if $u > s(t_0)$) we respectively obtain $\sup_M u = t_0$ and $\inf_M u = t_0$. The conclusion follows. 
\end{proof}

Turning our attention to graphs in $\HH^{m+1}$ along hyperspheres, we similary have the next result that extends \cite{docarmolawson}, where the minimal case is considered.

\begin{proposition}\label{prp_hyperspheres}
	Let $\psi_u : M \ra \HH^{m+1} = \RR \times_{\cosh t} \HH^{m}$ be an entire soliton graph with respect to $\cosh t \partial_t$ and with constant $c$.
	\begin{itemize}
		\item[$(i)$] If $c=0$, then $\psi_u$ is the totally geodesic hypersphere $\{t=0\}$;
		\item[$(ii)$] if $|c| > m/2$, there is no such $\psi_u$.
	\end{itemize}
\end{proposition}

\begin{proof}
	The soliton function is $\zeta_c(t) = m \sinh t + c \cosh^2t = c \sinh^2 t + m \sinh t + c$. With respect to the flow variable $s(t)$, $u$ is a bounded graph and thus we can apply Theorem \ref{gr:thm:BMPR} to deduce
	$$
	\zeta_c\big( t( \sup_M u) \big) \le 0 \le \zeta_c \big( t( \inf_M u) \big).
	$$
	In assumption $(i)$, $\zeta_c$ is strictly increasing and thus $u$ must be constant and correspond to the unique zero of $\zeta_c$. On the other hand, in $(ii)$  $\zeta_c$ is strictly positive or negative on $\RR$, leading to a contradiction. The conclusion follows.
\end{proof}

Next, we consider the case when the soliton function $\zeta_c(t)$ is negative along the graph. As expected, one needs much stronger conditions to deduce a Liouville type theorem and a consequent rigidity of the graph: again, parabolicity comes into play.

\begin{theorem} \label{teoparab}
	Let $\psi_u : M^m \to \RR \times_h \PP^m$ be a graphical mean curvature flow soliton with respect to the vector field $h(t)\partial_t$ with soliton constant $c$. Assume that $(\PP,\metric_{\PP})$ is complete and
	\begin{equation} \label{gr:parab}
	\frac{1}{\vol_{\PP}(\partial B_r)} \not\in L^1(+\infty)
	\end{equation}
	where $\partial B_r$ is the boundary of the geodesic ball in $(\PP,\metric_{\PP})$ of radius $r$ centered at a fixed origin $o\in\PP$. Suppose that $\psi(\PP)\subseteq[a,+\infty)\times\PP^m$ for some $a\in\RR$ and that $\zeta_c\leq 0$ in $[a,+\infty)$. Then $\psi(\PP)$ is a slice of the natural foliation of $\RR\times_h\PP^m$.
\end{theorem}

\begin{proof}
	We let $g=\log\sqrt{1+|Du|^2}\geq0$ where $u$ is defined in \eqref{gr:defw} and $D$ is the covariant derivative in $\PP$. Then
	\[
	\vol_{\PP,g}(\partial B_r) = \int_{\partial B_r}e^{-g} \leq \vol(\partial B_r)
	\]
	and assumption (\ref{gr:parab}) implies
	\begin{equation} \label{gr:fparab}
	\frac{1}{\vol_{\PP,g}(\partial B_r)} \not\in L^1(+\infty).
	\end{equation}
	Completeness of $(\PP,\metric_{\PP})$ and (\ref{gr:fparab}) imply, by Theorem 4.14 of \cite{AMR} that $\PP$ is parabolic with respect to the operator
	\begin{equation} \label{gr:fLap}
	e^g\div_{\PP}(e^{-g}D\;\cdot\;).
	\end{equation}
	On the other hand, from \eqref{eq_soliton_graph} we have the validity of
	\[
	e^g\div_{\PP}(e^{-g}Du) = f(u) = \zeta_c\big(s(u)\big) \leq 0
	\]
	since $\psi(\PP)\subseteq[a,+\infty)\times\PP$ and $\zeta_c\leq0$ on $[a,+\infty)$. Now $\pi_I \circ \psi$ is bounded below and therefore $u$ is bounded below, superharmonic with respect to the operator (\ref{gr:fLap}). It follows that $u$ is constant, thus $\psi$ is a soliton slice.
\end{proof}

As a direct corollary we have the following result, to be compared with Corollary \ref{stab:crl:hspaceumbilicity} and the subsequent Remark \ref{remhoro}.

\begin{corollary}
	The only entire graph $\psi_u : M \to \HH^3 = \RR \times_{e^t} \RR^2$ in the $3$-dimensional hyperbolic space that is a graphical soliton with respect to $e^t\partial_t$ with soliton constant $c < 0$ and contained in the half-space $\left[ \log\left(-\frac{2}{c}\right), + \infty\right)$ is the horosphere $\{t = \log\left(-\frac{2}{c}\right) \}$.
\end{corollary}

\begin{proof}
	It is sufficient to observe that the slice $\RR^2$ satisfies \eqref{gr:parab} and apply Theorem \ref{teoparab}.
\end{proof}

Translating solitons with controlled growth in products $\RR\times\PP$ have been studied in \cite{BMPR}, as well as self-expanders in Euclidean space. We state without proof the following theorems, that well fit with our results.

\begin{theorem}[\cite{BMPR}, Thm. 1.23]
	Let $(\PP, \metric_\PP)$ be complete manifold, set $r(x)=\dist_{\PP}(x,o)$ for some fixed origin $o \in \PP$ and suppose that for some $\sigma \in [0,2]$ either
	$$
	\begin{array}{ll}
	\disp\liminf_{r \ra \infty} \dfrac{\log\vol(B_r)}{r^{2-\sigma}} < \infty & \qquad \text{if } \, \sigma \in [0,2) \qquad \text{or} \\[0.5cm]
	\disp\liminf_{r \ra \infty} \dfrac{\log\vol(B_r)}{\log r} < \infty & \qquad \text{if } \, \sigma = 2.
	\end{array}
	$$
	Then, there exist no entire graph $\psi_u : M \to \RR \times \PP$ that is a translator with respect to the vertical direction $\partial_t$ and satisfies
	$$
	|u(x)| = o\big( r(x)^\sigma \big) \qquad \text{as } \, r(x) \ra \infty.
	$$
\end{theorem}

\begin{theorem}[\cite{BMPR}, Thm. 5.20]
	Let $\psi_u : M \to \RR \times \RR^m$ be a translator in $\RR^{m+1}$ with respect to a parallel vector $Y$. Assume that 
	$$
	\limsup_{r(x) \ra \infty} \frac{|u(x)|}{r(x)} = \hat u < \infty.
	$$
	Then, the angle $\vartheta$ between $Y$ and the horizontal hyperplane $\RR^m$ must satisfy
	$$
	\tan \vartheta \leq \hat u.
	$$
	In particular, if $\hat u = 0$, $\Sigma$ cannot be a self-translator with respect to a vector $Y$ which is not tangent to the horizontal $\RR^m$.
\end{theorem}

\begin{theorem}[\cite{BMPR}, Thm. 5.23]
	The only entire bounded graph in $\RR^{m+1}$ that is a self-expander for the mean curvature flow is a hyperplane passing through the origin.
\end{theorem}

We conclude this section with the next observation.

\begin{proposition} \label{gr:prp:unboundgrad}
	Let $\psi_u : M \to \RR \times \PP$ be a translating mean curvature flow soliton graph with soliton constant $c>0$. Let $(\PP,\metric_{\PP})$ be complete and with subexponential volume growth. Then
	\begin{equation} \label{gr:unboundgrad}
	\sup_{\PP}|Du| = +\infty.
	\end{equation}
\end{proposition}

\begin{remark}
	Observe that condition (\ref{gr:unboundgrad}) is obviously equivalent to
	\begin{equation}
	\inf_{\PP} H = \inf_{\PP} \frac{c}{m\sqrt{1+|Du|^2}} = 0.
	\end{equation}
\end{remark}

\begin{proof}[Proof of Proposition \ref{gr:prp:unboundgrad}]
	Suppose by contradiction that (\ref{gr:unboundgrad}) is false. Then, there exists $\Lambda>0$ such that $|Du|<\Lambda$ on $\PP$. Let $B_r=B_r(o)$ be the geodesic ball of radius $r$ centered at a fixed origin $o\in\PP$. From equation \eqref{eq_soliton_graph} satisfied by $u$ using the divergence theorem we deduce
	\[
	\Lambda \, \frac{\vol(\partial B_r)}{\sqrt{1+\Lambda^2}} \geq \int_{\partial B_r}\left\langle\frac{Du}{\sqrt{1+|Du|^2}},\nu\right\rangle = \int_{ B_r}\frac{c}{\sqrt{1+|Du|^2}} \geq c \, \frac{\vol(B_r)}{\sqrt{1+\Lambda^2}}
	\]
	and therefore, since $c>0$,
	\begin{equation} \label{gr:logvol}
	\frac{\vol(\partial B_r)}{\vol(B_r)} \geq \frac{c}{\Lambda} > 0.
	\end{equation}
	Integrating the left hand side of (\ref{gr:logvol}) we obtain that the volume growth of geodesic balls centered at $o$ is at least exponential, contradiction.
\end{proof}

\end{document}